\documentclass[reqno]{amsart}
\usepackage{amssymb}
\usepackage{amsmath}
\usepackage{amsfonts}
\usepackage[dvips]{graphicx}

\theoremstyle{plain}
\newtheorem{thm}{Theorem}

\newtheorem{lemma}{Lemma}
\theoremstyle{definition}

\newtheorem{definition}{Definition}
\newtheorem{proposition}{Proposition}

\newtheorem{example}{Example}
\theoremstyle{remark}

\newtheorem{remark}{Remark}

 \DeclareMathOperator\hdim{\dim_H}
\numberwithin{equation}{section}

\begin{document}
\title[Rigidity Theorem of Graph-directed Fractals]{Rigidity Theorem of
Graph-directed Fractals}
\author{Li-Feng Xi$^{\ast}$}
\address{Institute of Mathematics, Zhejiang Wanli University, Ningbo,
Zhejiang, 315100, P.~R. China}
\email{xilifengningbo@yahoo.com}
\author{Ying Xiong}
\address{Department of Mathematics, South China University of Technology,
Guangzhou, 510641, P.~R. China}
\email{xiongyng@gmail.com}
\subjclass[2000]{28A80}
\keywords{fractal, bilipschitz equivalence, graph-directed sets,
self-similar set}

\begin{abstract}
In this paper, we identify two fractals if and only if they are bilipschitz
equivalent. Fix ratio $r,$ for dust-like graph-directed sets with ratio $r$
and integer characteristics, we show that they are rigid in the sense that
they are uniquely determined by their Hausdorff dimensions. Using this
rigidity theorem, we show that in some class of self-similar sets, two
totally disconnected self-similar sets without complete overlaps are
bilipschitz equivalent.
\end{abstract}

\thanks{$^{\ast }$Corresponding author. This work is supported by NSF of China (Nos. 11071224, 11101159), NCET and
Morningside Center of Mathematics.}
\maketitle

\section{Introduction}

Falconer and Marsh \cite{FalMa92} pointed out that \textquotedblleft
topology\textquotedblright\ may be regarded as the study of equivalence
classes of sets under \emph{homeomorphism}, and \textquotedblleft fractal
geometry\textquotedblright\ is thought of as the study of equivalence
classes of fractals under \emph{bilipschitz mappings}.

We \emph{identify} two metric spaces $(B,$d$_{B})$ and $(B^{\prime },$d$%
_{B^{\prime }})$ if and only if they are bilipschitz equivalent, i.e., there
is a bijection $f:(B,$d$_{B})\rightarrow (B^{\prime },$d$_{B^{\prime }})$
such that for all $x,y\in B,$ $C^{-1}$d$_{B}(x,y)\leq $d$_{B^{\prime
}}(f(x),f(y))\leq C$d$_{B}(x,y),$ where $C>0$ is a constant.

Falconer and Marsh \cite{FalMa89} obtained a \textbf{rigidity theorem }on%
\textbf{\ }quasi-circles: they are uniquely determined by their Hausdorff
dimensions, i.e., two quasi-circles are bilipschitz equivalent if and only
if they have the same Hausdorff dimension.

\subsection{A class of graph-directed fractals}

\

Fix $r\in (0,1).$ Let $G$ be a directed graph, consisting of $p$ vertices $%
\{1,\cdots ,p\}$ and directed edges among these vertices. Suppose $(X,$d$)$
is a compact metric space such that for each edge $e$ there is a contracting
similitude $S_{e}:X\rightarrow X$ with ratio $r$, i.e., d$%
(S_{e}x_{1},S_{e}x_{2})=r$d$(x_{1},x_{2})$ for all $x_{1},x_{2}\in X$. Let $%
\mathcal{E}_{i,j}$ be the set of all the edges from vertex $i$ to vertex $j\
$and $A_{p\times p}=(a_{i,j})_{1\leq i,j\leq p}$ the \emph{adjacent matrix}
of $G$ defined by $a_{i,j}=\#\mathcal{E}_{i,j}.$ By \cite{MW} there is a
unique family $\{K_{1},\cdots ,K_{p}\}$ of compact sets in $\mathbb{(}X,$d$%
\mathbb{)}$ such that
\begin{equation}
K_{i}=\bigcup\nolimits_{j}\bigcup\nolimits_{e\in \mathcal{E}%
_{i,j}}S_{e}(K_{j}).\text{ }  \label{MWW}
\end{equation}%
We say that $\{K_{i}\}_{i}$ are \emph{dust-like }graph-directed sets with
ratio $r$, if the right hand of (\ref{MWW}) is a disjoint union for every $i$%
. We also say that $\{K_{i}\}_{i}$ have integer characteristic $m\geq 2$, if
there is a positive vector $v>0$ such that%
\begin{equation}
Av=mv.  \label{assumption2}
\end{equation}

In this paper, we try to obtain the rigidity theorem on a class $\mathfrak{A}%
^{r}$ of graph-directed fractals, where $\mathfrak{A}^{r}=\cup
_{m=2}^{\infty }\mathfrak{A}_{m}^{r},$ and $\mathfrak{A}_{m}^{r}=\{K:K=K_{1}$
for some dust-like graph-directed sets $\{K_{i}\}_{i}$ with ratio $r$ and
integer characteristic $m$\}.

\begin{remark}
It will be proved in Section 2 that assumption $(\ref{assumption2})$ is the
equivalent to%
\begin{equation}
0<\mathcal{H}^{s}(K_{1}),\cdots ,\mathcal{H}^{s}(K_{p})<\infty \text{ with }%
s=-\log m/\log r,\text{ }m(\geq 2)\in \mathbb{N}.  \label{assumption1}
\end{equation}
\end{remark}

\begin{remark}
The directed graph is said to be \emph{transitive} if for any vertices $i,j,$
there exists a directed path beginning at $i$ and ending at $j.$ In this
case, the adjacent matrix $A$ is \emph{irreducible}. By Perron-Frobenius
Theorem of non-negative matrix, assumption (\ref{assumption2}) on \emph{%
transitive} graph is equivalent to the following assumption: the
Perron-Frobenius eigenvalue of $A$ is an integer $m\geq 2.$
\end{remark}

Let $\Sigma _{m}^{r}=\{1,2,\cdots ,m\}^{\infty }$ be the symbolic system
equipped with the metric
\begin{equation*}
d(x_{0}x_{1}\cdots ,y_{0}y_{1}\cdots )=r^{\min \{k|\text{ }x_{k}\neq
y_{k}\}}.
\end{equation*}%
Then $\Sigma _{m}^{r}$ is a self-similar set generated by $m$ similitudes of
ratio $r.$

Our main result, the rigidity theorems on $\mathfrak{A}^{r}$, are stated as
follows.

\begin{thm}
If $\{K_{i}\}_{i=1}^{p}\ $are dust-like graph-directed sets with ratio $r$
and integer characteristic $m$, then $K_{1},\cdots ,K_{p}$ are bilipschitz
equivalent to $\Sigma _{m}^{r}$.
\end{thm}

\begin{thm}
The graph-directed fractals in $\mathfrak{A}^{r}$ are rigid in the sense
that they are uniquely determined by their Hausdorff dimensions or
characteristics. Furthermore, $K\in \mathfrak{A}_{m_{1}}^{r_{1}}\ $and $%
K^{\prime }\in \mathfrak{A}_{m_{2}}^{r_{2}}$ are bilipschitz equivalent if
and only if $r_{1}^{k_{1}}=r_{2}^{k_{2}},m_{1}^{k_{1}}=m_{2}^{k_{2}}$ for
some integers $k_{1},k_{2}\geq 1$.
\end{thm}

\subsection{Bilipschitz equivalence of self-similar fractals}

$\ $

An important fact is that two self-similar sets with the same dimension need
\emph{not} be bilipschitz equivalent. For example, let $3r^{\log 2/\log
3}=1, $ and a self-similar set generated by $S_{1}(x)=rx,$ $%
S_{2}(x)=rx+(1-r)/2,$ $S_{3}(x)=rx+1-r.$ Then this self-similar set and the
Cantor ternary set have the same Hausdorff dimension$\ \log 2/\log 3,$ but
they are not bilipschitz equivalent, please refers to Cooper and Pignataro
\cite{CooPi88}, David and Semmes \cite{DavSe97} and Falconer and Marsh \cite%
{FalMa92}.

Falconer and Marsh~\cite{FalMa92} gave a necessary condition for
self-similar sets satisfying the strong separation condition (or \textbf{%
without overlaps}) to be bilipschitz equivalent, and a necessary and
sufficient condition was also obtained by Xi~\cite{Xi10}. Please refers to
Deng and Wen et al \cite{DWXX}, Llorente and Mattila \cite{LM}, Mattila and
Saaranen \cite{MS} and Rao, Ruan and Wang \cite{RRW}.

For self-similar sets \textbf{with overlaps}, we will show two cases:
self-similar arcs and $\{1,3,5\}$-$\{1,4,5\}$ problem with its
generalization. Wen and Xi \cite{WenXi03} showed that the rigidity is
invalid on self-similar arcs, in fact they constructed two self-similar
arcs, with the same Hausdorff dimension, which are not bilipschitz
equivalent. 
Define two self-similar sets respectively by $F_{i}=F_{i}/5\cup
((i+1)/5+F_{i}/5)\cup (4/5+F_{i}/5)\ (i=1,2).$ Then both $F_{1}$ and $F_{2}$
satisfy the open set condition and have the same Hausdorff dimension $\log
3/\log 5$. David and Semmes \cite{DavSe97} asked whether they are
bilipschitz equivalence, and the question is called \textquotedblleft $%
\{1,3,5\}$-$\{1,4,5\}$ problem". \cite{RaRuX06} gave an affirmative answer
to the problem. Furthermore, Xi and Xiong \cite{XiXiong2} answered the
problem in high dimensional Euclidean spaces: Let $B,B^{\prime }\subset
\{0,1,\cdots ,(n-1)\}^{l}.$ Given two totally disconnected self-similar sets
\begin{equation}
E_{B}=\bigcup\nolimits_{b\in B}(\frac{E_{B}}{n}+\frac{b}{n})\text{ and }%
E_{B^{\prime }}=\bigcup\nolimits_{b^{\prime }\in B^{\prime }}(\frac{%
E_{B^{\prime }}}{n}+\frac{b^{\prime }}{n}),  \label{total}
\end{equation}%
then $E_{B}$ and $E_{B^{\prime }}$ are bilipschitz equivalent if and only if
$\#B=\#B^{\prime }.$


\subsection{Self-similar fractals with overlaps}

$\ $\

Self-similar sets with overlaps have very \emph{complicated }structures. For
example, the open set condition (OSC), which means the overlaps are little,
was introduced by Moran \cite{Moran} and studied by Hutchinson \cite{H}.
Schief \cite{Schief}, Bandt and Graf \cite{B2} showed the relation between
the open set condition and the positive Hausdorff measure. Falconer \cite{F}
proved some \textquotedblleft generic\textquotedblright\ result on Hausdorff
dimension of self-similar sets without the assumption about the open set
condition. One useful notion \textquotedblleft
transversality\textquotedblright\ to study self-similar sets (or measures)
with overlaps can be found e.g. in Keane, Smorodinsky and Solomyak \cite{K},
Pollicott and Simon \cite{PS}, Simon and Solomyak \cite{Simon} and Solomyak
\cite{Solo}. Feng and Lau \cite{FL}, Lau and Ngai \cite{LN} studied the weak
separation condition. Please refers to Bandt and Hung \cite{B4}, Sumi \cite%
{Sumi} for some recent work.

For self-similar set $E_{\lambda }=E_{\lambda }/3\cup (E_{\lambda
}/3+\lambda /3)\cup (E_{\lambda }/3+2/3)$, a conjecture of Furstenberg says
that $\hdim E_{\lambda }=1$ for any $\lambda $ irrational. \'{S}wi\c{a}tek
and Veerman \cite{SV} proved that $\hdim E_{\lambda }>0.767$ for every $%
\lambda $ irrational. Kenyon \cite{K2}, Rao and Wen \cite{RW} obtained that $%
\mathcal{H}^{1}(E_{\lambda })>0$ if and only if $\lambda =p/q\in \mathbb{Q}$
with $p\equiv q\not\equiv 0($mod$3).$ The key idea of \cite{RW} is
\textquotedblleft graph-directed struture\textquotedblright\ introduced by
Mauldin and Williams \cite{MW}. In particular, Rao and Wen \cite{RW} and Wen
\cite{Wen}\ studied the self-similar sets with \textquotedblleft complete
overlaps\textquotedblright .

\begin{definition}
We say that a self-similar set $E=\cup _{i}S_{i}(E)$ has complete overlaps,
if there are two distinct sequences $i_{1}\cdots i_{t},j_{1}\cdots
j_{t^{\prime }}$ such that%
\begin{equation*}
S_{i_{1}}\circ \cdots \circ S_{i_{t}}\equiv S_{j_{1}}\circ \cdots \circ
S_{j_{t^{\prime }}}.
\end{equation*}
\end{definition}

\begin{remark}
For example, the self-similar sets in $(\ref{total})$ has no complete
overlaps.
\end{remark}

\subsection{A class of self-similar fractals in $\mathfrak{A}^{1/n}$}

$\ $

Suppose $\Gamma $ is a discrete additional group in $\mathbb{R}^{l},$ and $%
\mathbb{G}$ is a finite isometric group on $\Gamma ,$ i.e., $\mathbb{G}$ is
a finite group and for any $g\in \mathbb{G},$ we have $g:\Gamma \rightarrow
\Gamma $, $g(0)=0$ and $|g(a_{1})-g(a_{2})|=|a_{1}-a_{2}|$ for all $%
a_{1},a_{2}\in \Gamma ,$ which implies that $g$ can be extended to a \emph{%
linear} isometry of $\mathbb{R}^{l}.$ In particular, $g(B(0,\delta
))=B(0,\delta )$ for any closed ball with center $0$ and radius $\delta .$
Fix an integer $n\geq 2$. Consider the similitude
\begin{equation}
S(x)=g(x)/n+b\text{ where }g\in \mathbb{G}\text{ and }b\in \Gamma .
\label{similarity}
\end{equation}%
Let $\Lambda $ be the collection of all the self-similar sets in $\mathbb{R}%
^{l}$\ generated by contractive similitudes in the form of (\ref{similarity}%
).

Applying Theorem 1 to this class of self-similar sets, we have

\begin{thm}
Suppose $E=\cup _{i=1}^{m}(\frac{g_{i}E}{n}+b_{i})\in \Lambda $ is a totally
disconnected self-similar set without complete overlaps. Then $E\in
\mathfrak{A}^{1/n}$ with characteristic $m,$ and thus $E$ is bilipschitz
equivalent to $\Sigma _{m}^{n^{-1}}.$
\end{thm}

\vspace{-0.5cm} 
\begin{figure}[tbph]
\centering\includegraphics[width=0.55\textwidth]{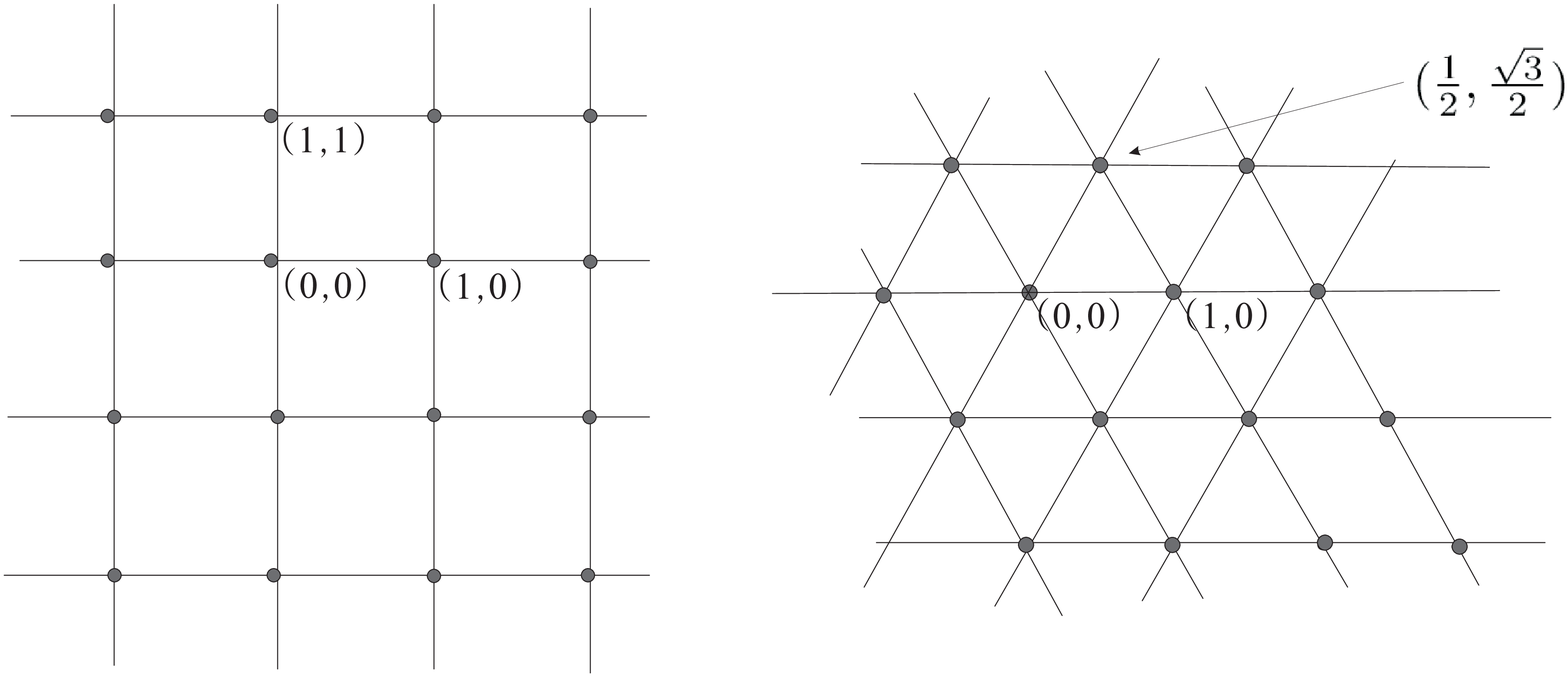} \vspace{-0.5cm}
\caption{Discrete groups on the plane}
\end{figure}
\vspace{-0.5cm}

\begin{example}
Let $n\geq 2,$ $\Gamma _{1}=\mathbb{Z}^{2}$ and $\Gamma _{2}=\{(\frac{a}{2},%
\frac{\sqrt{3}}{2}b):a,b\in \mathbb{Z}$ and $a\equiv b($mod2$)\},$ then $%
\Gamma _{i}$ are discrete additional groups as in Figure 1. The finite
isometric group with respect to $\Gamma _{2}$ has $12$ elements, including
rotations and reflections.
\end{example}

\vspace{-0.5cm}
\begin{figure}[tbph]
\centering\includegraphics[width=0.65\textwidth]{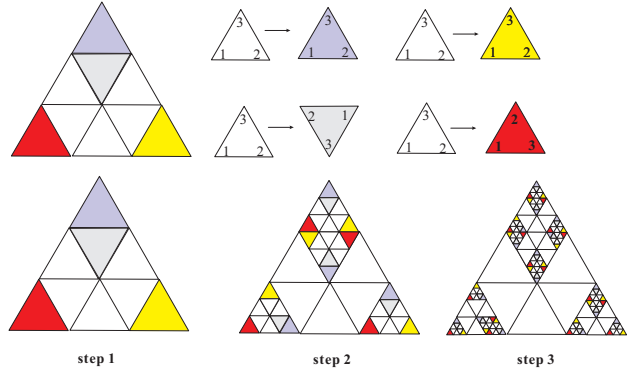} \vspace{-0.3cm}
\caption{Initial pattern and first three steps in construction}
\end{figure}
\vspace{-0.3cm}

\begin{example}
Let $n=3$ and $\Gamma =\Gamma _{2}$ as in Example 1. As in Figure 2, we take
4 small colored triangles and select the corresponding isometries. We also
show the first three steps in construction.
\end{example}

\begin{example}
Let $\Gamma =\mathbb{Z}^{l}$ with $l\geq 2$, then its isometric group
contains the element$:$
\begin{equation*}
g(x_{1},\cdots ,x_{l})=(s_{1}x_{\sigma (1)},s_{2}x_{\sigma (2)},\cdots
,s_{l}x_{\sigma (l)}),\text{ }
\end{equation*}%
where $\sigma $ is a permutation on $\{1,\cdots ,l\}$ and sign $s_{i}\in
\{-1,1\}$ for all $i.$ In this case, we notice that $\Lambda $ contains any
self-similar sets in $(\ref{total}).$
\end{example}

\begin{figure}[tbph]
\centering\includegraphics[width=0.9\textwidth]{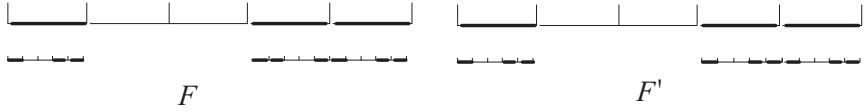} \vspace{-0.3cm}
\caption{Different structures in $F$ and $F^{\prime }$}
\end{figure}
\vspace{-0.3cm}

\begin{example}
For $n=5$ and $l=1,$ let
\begin{equation*}
F=\frac{F}{5}\cup \frac{-F+4}{5}\cup \frac{F+4}{5},\text{ }F^{\prime }=\frac{%
F^{\prime }}{5}\cup \frac{F^{\prime }+3}{5}\cup \frac{F^{\prime }+4}{5}.
\end{equation*}%
Due to the minus sign in the similarity $x\mapsto \frac{-x+4}{5},$ these two
self-similar sets are quite different (Figure 3). It is easy to show $F$ and
$F^{\prime }$ are totally disconnected self-similar sets without complete
overlaps. Then it follows from Theorem 4 that $F$ and $F^{\prime }$ are
bilipschitz equivalent.
\end{example}


We consider the self-similar set $E^{\ast }$ generated by
\begin{equation*}
S_{i}(x)=x/n+b_{i}\text{ with }b_{i}\in \mathbb{Q}^{l}\text{ for }i=1,\cdots
,m.
\end{equation*}
Let $\Gamma $ be the additional group generated by $\{b_{i}\}_{i}.$ Suppose $%
n>m,$ then $\dim _{H}E^{\ast }\leq \log m/\log n<1$ which implies $E^{\ast }$
is totally disconnected. Theorem 3 implies that if $E^{\ast }$ has no
complete overlaps, then $E^{\ast }$ is bilipschitz equivalent to $\Sigma
_{m.}^{1/n}$ Please see Example \ref{non} for self-similar set $E=\frac{E}{5}%
\cup (\frac{E}{5}+\frac{7}{10})\cup (\frac{E}{5}+\frac{8}{10}).$

\subsection{Algorithm in polynomial time to test complete overlaps in $%
\Lambda $}

$\ $

Given $E=\cup _{i=1}^{m}(\frac{g_{i}E}{n}+b_{i})\in \Lambda ,$ let $%
M=2\max_{1\leq i\leq m}|b_{i}|/(n-1).$

We will construct a directed graph with vertex set $\{x\in \Gamma :|x|\leq
M\}\times \mathbb{G}\times \mathbb{G}$. Set an edge from vertex $%
(x,g,g^{\prime })$ to vertex$\ (x_{1},g_{1},g_{1}^{\prime }),$ if and only
if there is a pair $(i,j)\in \{1,\cdots ,m\}^{2}$ such that $%
(x_{1},g_{1},g_{1}^{\prime })=(nx+g(b_{i})-g^{\prime
}(b_{j}),gg_{i},g^{\prime }g_{j}).$ We denote this edge by $(x,g,g^{\prime })%
\overset{(i,j)}{\rightarrow }(x_{1},g_{1},g_{1}^{\prime }).$ This is a
finite graph since $\Gamma $ is discrete and $\mathbb{G}$ is finite. We call
$(x,g,g^{\prime })$ an \emph{original vertex}, if $(x,g,g^{\prime
})=(b_{i}-b_{j},g_{i},g_{j})$ for some $i\neq j,$ and call $%
(x_{1},g_{1},g_{1}^{\prime })$ a \emph{boundary vertex}, if $x_{1}=0$ and $%
g_{1}=g_{1}^{\prime }.$ We have the following result.

\begin{proposition}
$E$ has complete overlaps if and only if there is a directed path starting
at an original vertex and ending at a boundary vertex.
\end{proposition}

Applying Dijkstra's algorithm \cite{Dijk}, in polynomial time we can test
whether there is a such directed path.

When $\mathbb{G}=\{id\},$ we notice that we don't need the group part $%
\mathbb{G\times G}$\ in $\{x\in \Gamma :|x|\leq M\}\mathbb{\times G\times G}%
. $ Consider the graph $\{x\in \Gamma :|x|\leq M\},$ there is an edge from
vertex $x$ to vertex$\ x,$ if and only if there is a pair $(i,j)\in
\{1,\cdots ,m\}^{2}$ such that $x_{1}=nx+b_{i}-b_{j}.$ We denote the edge by
$x\overset{(i,j)}{\rightarrow }x_{1}.$

\begin{example}
\label{non}Let $5=n>m=3$ and $E=\frac{E}{5}\cup (\frac{E}{5}+\frac{7}{10}%
)\cup (\frac{E}{5}+\frac{8}{10}).$ Here
\begin{equation*}
n=5,\mathbb{G}=\{id\},\Gamma =\mathbb{Z}/10,M=0.4
\end{equation*}%
and $\{x\in \Gamma :|x|\leq M\}=\{-0.4,-0.3,-0.2,-0.1,0,0.1,0.2,0.3,0.4\}.$
The original vertices are $\{-0.1,0.1\}$ and the unique boundary vertex is $%
\{0\}.$ If $x\overset{(i,j)}{\rightarrow }0$ for some $x\in \{x\in \Gamma
:|x|\leq M\},$ then $0=nx+b_{i}-b_{j},$ that is%
\begin{equation*}
nx=b_{j}-b_{i}.
\end{equation*}%
However, $b_{j}-b_{i}$ belongs to $P=\{-0.8,-0.7,-0.1,0,0.1,0.7,0.8\}\ $%
satisfying $P\cap \{5x:x\in \Gamma $ and $|x|\leq M\}=\{0\}.$ Therefore, $%
x=0.$ This means there is no directed path from original vertices to
boundary vertices. Then $E$ is totally disconnected self-similar set without
complete overlaps and thus $E$ is bilipschitz equivalent to $\Sigma
_{3}^{1/5}$.
\end{example}

\medskip

The paper is organized as follows. Section 2 is the preliminaries, including
the techniques of number theory, non-negative matrix and bilipschitz
equivalence. Section 3 is the proofs of Theorems 1 and 2. Section 4 is
devoted to Theorem 3, in which total disconnectedness implies the
graph-directed structure and the loss of complete overlaps insures the
integer characteristic. The last section indicates some related open
questions.

\medskip

\section{Preliminaries}

\subsection{Combinatorial lemma from Frobenius coin problem}

$\ $

It is easy to check that $\{2a_{1}+5a_{2}:$ $a_{1},a_{2}\in \mathbb{N}\cup
\{0\}\}=\{n:n>3\}.$ In general, we have the following Frobenius coin problem
\cite{R} and Frobenius number $\phi (a_{1},\cdots ,a_{s}):$ Given natural
integers $a_{1},\cdots ,a_{s}$ $(s\geq 2)$ with their greatest common
divisor $\gcd (a_{1},\cdots ,a_{s})=1,$ then there exists an integer $\phi
(a_{1},\cdots ,a_{s})$ such that for each integer $n>\phi (a_{1},\cdots
,a_{s})$ the following equation
\begin{equation*}
a_{1}x_{1}+\cdots +a_{s}x_{s}=n
\end{equation*}%
have integer solution $x_{1}\geq 0,\cdots ,x_{s}\geq 0.$

\begin{lemma}
\label{com} Let $a_{1},a_{2},\cdots ,a_{s}$ $(s\geq 2)$ be positive integers
with $\gcd (a_{1},\cdots ,a_{s})=1.$ Suppose $b,c\in \mathbb{N}$ with
\begin{equation}
b>2s(\max\nolimits_{t}a_{t})(\max\nolimits_{t}a_{t}+\phi (a_{1},\cdots
,a_{s})).  \label{p}
\end{equation}%
and $l_{\omega }\in \{a_{1},\cdots ,a_{s}\}$ for every $\omega \ $in an
index set $\Omega $ satisfying
\begin{equation}
\sum\nolimits_{\omega \in \Omega }l_{\omega }=b\cdot c.  \label{omega}
\end{equation}%
There is a decomposition of $\Omega =\cup _{j=1}^{s}\Delta _{j}$ such that $%
\Delta _{j}\cap \Delta _{j^{\prime }}=\emptyset $ for $j\neq j^{\prime },$
and $l_{\omega }=a_{j}$ for any $\omega \in \Delta _{j}$ and $1\leq j\leq s.$
If for every $j,$
\begin{equation}
\#\Delta _{j}\geq (\max\nolimits_{t}a_{t}+\phi (a_{1},\cdots ,a_{s}))\cdot c,
\label{condi}
\end{equation}%
then there is a decomposition $\Omega =\bigcup_{t=1}^{c}\Omega _{t}$ such
that for every $1\leq t\leq c$,
\begin{equation}
\sum\nolimits_{\omega \in \Omega _{t}}l_{\omega }=b.
\end{equation}
\end{lemma}

\begin{proof}
We will prove this inductive assumption on the integer $c\geq 1$.

(1)\ This is trivial when $c=1$.

(2)\ Suppose the inductive assumption on $c$ is true for $c=1,\cdots ,(k-1)$.

Now, set $c=k\geq 2.$ Take $\Delta _{j}^{\prime }\subset \Delta _{j}$ such
that $\#\Delta _{j}^{\prime }=k(\max_{t}a_{t}+\phi (a_{1},\cdots ,a_{s})).$
Let $\Omega ^{\ast }=\Omega \backslash \left( \cup _{j=1}^{s}\Delta
_{j}^{\prime }\right) .$ Then by (\ref{p}) and (\ref{omega}), we have
\begin{eqnarray*}
&&\sum_{\omega \in \Omega ^{\ast }}l_{\omega }=\sum_{\omega \in \Omega
}l_{\omega }-\sum\nolimits_{j=1}^{s}\sum_{\omega \in \Delta _{j}^{\prime
}}l_{\omega } \\
&=&b\cdot k-k\sum\nolimits_{j=1}^{s}(a_{j})(\max\nolimits_{t}a_{t}+\phi
(a_{1},\cdots ,a_{s})) \\
&\geq &b\cdot k-k\cdot s(\max\nolimits_{t}a_{t})(\max\nolimits_{t}a_{t}+\phi
(a_{1},\cdots ,a_{s})) \\
&\geq &k[b-s(\max\nolimits_{t}a_{t})(\max\nolimits_{t}a_{t}+\phi
(a_{1},\cdots ,a_{s}))] \\
&\geq &k(b/2)\geq b.
\end{eqnarray*}%
Since
\begin{equation*}
\sum_{\omega \in \Omega ^{\ast }}l_{\omega }\geq b\text{ and }b-\phi
(a_{1},\cdots ,a_{s})>\max_{t}a_{t},
\end{equation*}%
we can select a \emph{maximal }subset $\Theta $ of $\Omega ^{\ast }$ such
that
\begin{equation*}
\sum\nolimits_{\omega \in \Theta }l_{\omega }<b-\phi (a_{1},\cdots ,a_{s}).
\end{equation*}
We conclude that
\begin{equation*}
\sum\nolimits_{\omega \in \Theta }l_{\omega }\geq b-\phi (a_{1},\cdots
,a_{s})-\max\nolimits_{t}a_{t}.
\end{equation*}%
Otherwise, we notice that $\Omega ^{\ast }\backslash \Theta $ is nonempty
since $\sum_{\omega \in \Omega ^{\ast }}l_{\omega }\geq b$ and $%
\sum\nolimits_{\omega \in \Theta }l_{\omega }<b,$ and thus we can take $%
i_{0}\in \Omega ^{\ast }\backslash \Theta $ which implies
\begin{eqnarray*}
\sum\nolimits_{\omega \in \Theta }l_{\omega }<\sum\nolimits_{\omega \in
\Theta \cup \{i_{0}\}}l_{\omega } &<&(b-\phi (a_{1},\cdots
,a_{s})-\max\nolimits_{t}a_{t})+\max\nolimits_{t}a_{t} \\
&=&b-\phi (a_{1},\cdots ,a_{s}),
\end{eqnarray*}%
which contradicts the maximality of $\Theta .$ Now,
\begin{equation}
\phi (a_{1},\cdots ,a_{s})<b-\sum\nolimits_{\omega \in \Theta }l_{\omega
}\leq \phi (a_{1},\cdots ,a_{s})+\max\nolimits_{t}a_{t}.  \label{phi}
\end{equation}%
Applying Frobenius number $\phi (a_{1},\cdots ,a_{s})$ to (\ref{phi}), we
can find non-negative integers $x_{1},\cdots ,x_{s}$ such that
\begin{equation}
a_{1}x_{1}+\cdots +a_{s}x_{s}=b-\sum\nolimits_{\omega \in \Theta }l_{\omega
},  \label{sum}
\end{equation}%
where $x_{j}\leq \phi (a_{1},\cdots ,a_{s})+\max_{t}a_{t}$ for all $j.$ Take
$\Delta _{j}^{\ast }\subset \Delta _{j}^{\prime }$ with $\#\Delta _{j}^{\ast
}=x_{j}$ for every $j.$ Let $\Omega _{1}=\Theta \cup \left( \cup
_{j=1}^{s}\Delta _{j}^{\ast }\right) $. By (\ref{sum}), we have
\begin{equation*}
\sum\nolimits_{\omega \in \Omega _{1}}l_{\omega }=b.
\end{equation*}%
Let $\Omega ^{\prime }=\Omega \backslash \Omega _{1},$ then
\begin{equation*}
\sum\nolimits_{\omega \in \Omega ^{\prime }}l_{\omega }=b(k-1)
\end{equation*}%
and for $j=1,\cdots ,s,$ we have%
\begin{eqnarray*}
\#(\Delta _{j}^{\prime }\backslash \Delta _{j}^{\ast })
&=&(\max\nolimits_{t}a_{t}+\phi (a_{1},\cdots ,a_{s}))k-x_{j} \\
&\geq &(\max\nolimits_{t}a_{t}+\phi (a_{1},\cdots ,a_{s}))(k-1).
\end{eqnarray*}%
For $\Omega ^{\prime },$ using the inductive assumption on $c=k-1,$ we get
this lemma.
\end{proof}

\subsection{Connected block of directed Graph\label{order of block}}

$\ $

Given a directed graph $G,$ we give the equivalent relation on the vertex
set, let $i\sim i$ and for vertices $i$ and $j\ $with $i\neq j,$%
\begin{equation*}
i\sim j\text{ iff there are direct pathes from }i\text{ to }j\text{ and from
}j\text{ to }i\text{ respectively.}
\end{equation*}%
We call the equivalence class $[i]=\{j:j\sim i\}$ a \emph{connected block}.

Define a \emph{partial order} among connected block as follows%
\begin{equation*}
\lbrack j]\prec \lbrack j^{\prime }]\text{ iff there is a directed path from
a vertex in }[j]\text{ to a vertex in }[j^{\prime }].
\end{equation*}%
Under this partial order, we take the maximal elements from the finite
connected blocks and denote%
\begin{equation*}
\mathfrak{B}_{0}=\{[i]:[i]\text{ is maximal}\}.
\end{equation*}%
Inductively, let
\begin{equation*}
\mathfrak{B}_{k}=\{[i]:[i]\text{ is maximal in the complement of }\cup
_{q\leq k-1}\mathfrak{B}_{q}\}.
\end{equation*}%
We say that a vertex has rank $k,$ if this vertex belongs to a connected
block in $\mathfrak{B}_{k}.$ By the definition of $\mathfrak{B}_{k},$ we have

\begin{lemma}
\label{lemma:order}Given a vertex $i$ of rank $k,$ then any directed path
starting at this vertex will end at some vertex of rank $q$ with $q\leq k-1\
$or a vertex $j$ of rank $k$ with $i\sim j.$
\end{lemma}

\subsection{Integer characteristic}

$\ $

\begin{lemma}
Assumptions $(\ref{assumption2})$ and $(\ref{assumption1})$ on integer
characteristic are equivalent.
\end{lemma}

\begin{proof}
If assumption (\ref{assumption1}) holds, then by the definition of adjacent
matrix, we have%
\begin{equation*}
\mathcal{H}^{s}(K_{i})=r^{s}\sum\nolimits_{j}a_{ij}\mathcal{H}^{s}(K_{j})%
\text{ for every }i,
\end{equation*}%
where $s=-\log m/\log r$. Therefore, $r^{s}=1/m,$ i.e.,
\begin{equation*}
Av=mv
\end{equation*}%
with $v=(\mathcal{H}^{s}(K_{1}),\cdots ,\mathcal{H}^{s}(K_{p}))^{T}>0.$

\medskip

If assumptions (\ref{assumption2}) holds, at first we will shows
\begin{equation*}
\mathcal{H}^{s}(K_{i})>0\text{ for any }i.
\end{equation*}%
If $i$ is a vertex of rank $0$ defined above$,$ then $i$ belongs to a
connected block with its adjacent matrix $B_{k_{2}\times k_{2}}$ is
irreducible. By a permutation on the vertices $\{1,\cdots ,p\},$ we can
assume $Av=mv$ for $v=(v_{1},\cdots ,v_{p})>0,$ where $A$ is a upper block
triangular matrix in the form
\begin{equation*}
A=\left(
\begin{array}{cc}
C_{k_{1}\times k_{1}} & D_{k_{1}\times k_{2}} \\
O & B_{k_{2}\times k_{2}}%
\end{array}%
\right) \text{ with }k_{1}+k_{2}=p.
\end{equation*}%
Then for irreducible matrix $B_{k_{2}\times k_{2}}$,%
\begin{equation*}
B(v_{k_{1}+1},\cdots ,v_{p})^{T}=m(v_{k_{1}+1},\cdots ,v_{p})^{T}\text{ with
}(v_{k_{1}+1},\cdots ,v_{p})^{T}>0.
\end{equation*}%
By Perron-Frobenius Theorem on irreducible matrix, we have that $m$ is the
spectrum radius of $B.$ Using Corollary 3.5 of \cite{Fal97} on dust-like
graph-directed sets with respect to the transitive graph, for the block $[i]$
of rank $0,$ we have%
\begin{equation*}
0<\mathcal{H}^{s}(K_{j})<\infty \text{ for any }j\sim i\ \text{with }s=-\log
m/\log r.
\end{equation*}%
Suppose there is an edge $e$ from $j_{1}$ to $j_{2},$ then $%
S_{e}(K_{j_{2}})\subset (K_{j_{1}}),$ which implies
\begin{equation}
\mathcal{H}^{s}(K_{j_{2}})>0\Rightarrow \mathcal{H}^{s}(K_{j_{1}})>0\text{
if there is an edge from }j_{1}\text{ to }j_{2}.  \label{imply}
\end{equation}%
Inductively, we assume $\mathcal{H}^{s}(K_{j})>0$ for vertex of rank $k.$
Now, let $i$ be a vertex of rank $k+1.$ For connected block $[i],$ there
exist a vertex $i_{0}$ in $[i]$ and a vertex $j$ of rank less than or equal
to $k$ such that there is a path from $i$ to $i_{0}$ and a edge from $i_{0}$
to $j.$ Therefore, it follows from (\ref{imply}) and the inductive
assumption that $\mathcal{H}^{s}(K_{i})>0$ for vertex of rank $k+1.$

\medskip

Now, we will prove that $\mathcal{H}^{s}(K_{j})<\infty $ for any vertex $j$.

Take $v>0$ such that $v>(1,\cdots 1)^{T}$ and $Av=mv.$ Let $%
A^{k}=(a_{ij}^{(k)})_{i,j}.$ Then
\begin{equation*}
\sum\nolimits_{i,j}a_{ij}^{(k)}=(1,\cdots 1)A^{k}(1,\cdots 1)^{T}\leq
(1,\cdots 1)A^{k}v\leq ((1,\cdots 1)v)\cdot m^{k}.
\end{equation*}%
For any $k,$ then $K_{i}=\bigcup\nolimits_{j}\bigcup\nolimits_{e^{\ast }\in
\mathcal{E}_{i,j}^{k}}S_{e^{\ast }}(K_{j}),$ where $\mathcal{E}_{i,j}^{k}$
is the collection of all paths of length $k$ starting at $i$ and ending at $%
j. $ Therefore, for each $i,$ we have%
\begin{eqnarray*}
\mathcal{H}^{s}(K_{i}) &\leq &\lim_{k\rightarrow \infty }\left(
\sum\nolimits_{j}a_{ij}^{(k)}\right) \cdot (r^{k})^{s}\max_{j}|K_{j}| \\
&\leq &\lim_{k\rightarrow \infty }\left(
\sum\nolimits_{i,j}a_{ij}^{(k)}\right) \cdot (r^{k})^{s}\max_{j}|K_{j}| \\
&\leq &((1,\cdots 1)v)\cdot \max_{j}|K_{j}|<\infty .
\end{eqnarray*}%
where $|K_{j}|$ is the diameter of $K_{j}$.

Hence assumption (\ref{assumption1}) holds.
\end{proof}

\subsection{Non-negative matrix}

$\ $

A non-negative matrix $A$ is said to be primitive, if $A^{h}>0$ for some
natural integer $h.$ We say a matrix $B$ is \emph{irreducible}, if $B$
cannot be conjugated into block upper triangular form by a permutation
matrix $P$:%
\begin{equation*}
PBP^{-1}\neq \left(
\begin{array}{cc}
D_{1} & D_{2} \\
0 & D_{3}%
\end{array}%
\right) .
\end{equation*}%
The following lemma can be found in \cite{BP}.

\begin{lemma}
\label{spec}Let $B_{n\times n}$ be a non-negative irreducible matrix, and $%
\rho (B)>0$ its Perron-Frobenius eigenvalue. For any positive vector $%
x=(x_{1},\cdots ,x_{n})^{T}>0,$ we have $\rho (B)\leq \max_{i}\frac{(Bx)_{i}%
}{x_{i}}.$ Here $\rho (B)=\max_{i}\frac{(Bx)_{i}}{x_{i}}$ if and only if $%
Bx=\rho (B)x.$
\end{lemma}

For any irreducible matrix $B,$ Corollary 3.2.3B of \cite{LL} indicates
that: If a non-negative matrix $B$ is irreducible, then there is a
permutation matrix $P$ and a natural integer $d$ such that
\begin{equation*}
PB^{d}P^{-1}=diag(B_{1},\cdots ,B_{d}),
\end{equation*}%
where each $B_{i}$ is primitive. Therefore, for each $B_{i}$ there is a
positive integer $h_{i}$ such that $(B_{i})^{h_{i}}>0.$ Take $%
u=d\prod\nolimits_{i=1}^{d}h_{i},$ we have

\begin{lemma}
\label{irreducibe} If a non-negative matrix $B$ is irreducible, then there
is a permutation matrix $P$ and a natural integer $u$ such that
\begin{equation*}
PB^{u}P^{-1}=diag(D_{1},\cdots ,D_{d}),
\end{equation*}%
where $D_{1},\cdots ,D_{d}>0$ are positive square matrices.
\end{lemma}

\subsection{Bilipschitz equivalence}

$\ $

We say that a bijection $f$ form a metric space $(B,d_{B})$ to another $%
(B^{\prime },d_{B^{\prime }})$ is bilipschitz mapping with bilipschitz
constant $blip(f)\geq 1,$ if
\begin{equation}
blip(f)=\inf \{c\geq 1:c^{-1}\leq \frac{\text{d}_{B^{\prime }}(f(x),f(y))}{%
\text{d}_{B}(x,y)}\leq c\text{ for all }x\neq y\}
\end{equation}

The next lemma follows from~\cite{CooPi88,FalMa92} or Proposition~11.8 of
\cite{DavSe97}.

\begin{lemma}
\label{sym} $\Sigma _{n_{1}}^{r_{1}}$ and $\Sigma _{n_{2}}^{r_{2}}$ are
bilipschitz equivalent if and only if there are $k_{1},\,k_{2}\in \mathbb{N}$
such that $n_{1}^{k_{1}}=n_{2}^{k_{2}}$ and $r_{1}^{k_{1}}=r_{2}^{k_{2}}$.
\end{lemma}

Theorem 2.1 of~\cite{RaRuX06} yields the following lemma.

\begin{lemma}
\label{graph} Suppose $\left\{ K_{i}\right\} _{i=1}^{p}$ and $\left\{
K_{i}^{\prime }\right\} _{i=1}^{p}$ are dust-like graph-directed sets on the
same graph $G$ satisfying $(\ref{MWW})$ and $K_{i}^{\prime
}=\bigcup_{j}\bigcup_{e\in \mathcal{E}_{i,j}}S_{e}^{\prime }(K_{j}^{\prime
}) $, where $S_{e}^{\prime }:(X^{\prime },$d$^{\prime })\rightarrow
(X^{\prime },$d$^{\prime })$ for some compact metric space $(X^{\prime },$d$%
^{\prime })$. If for each edge $e$ the corresponding similarities $S_{e}$
and $S_{e}^{\prime }$ have the same ratio $\rho _{e}$, then $K_{i}$ and $%
K_{i}^{\prime }$ are bilipschitz equivalent for each $i$.
\end{lemma}

The next lemma comes from \cite{CooPi88}.

\begin{lemma}
\label{copy}Suppose $E$ is a dust-like self-similar set in a compact metric
space. Let $K=\cup _{i}f_{i}(E)$ be a disjoint union such that $f_{i}$ is a
bilipschitz mapping for each $i.$ Then $K$ and $E$ are bilipschitz
equivalent.
\end{lemma}

\subsection{Connectedness}

$\ $

Recall Lemma 2.3 and Lemma 2.4 of \cite{XiXiong2} as follows.

\begin{lemma}
\label{l:connected}Let $Y$ be a compact subset of $\mathbb{R}^{l}.$ Suppose $%
\{X_{k}\}_{k}$ are connected compact subsets of $Y$. Then there exist a
subsequence $\{k_{i}\}_{i}$ and a connected X such that $X_{k_{i}}\overset{%
\text{d}_{H}}{\longrightarrow }X$ as $i\rightarrow \infty $, where $d_{H}$
is the Hausdorff metric.
\end{lemma}

\begin{lemma}
\label{l:total}Let $\{Y_{1},Y_{2},\cdots ,Y_{k}\}$ be totally disconnected
compact subsets of $\mathbb{R}^{l}$. Then $Y=\cup _{i=1}^{k}Y_{t}$ is
totally disconnected.
\end{lemma}

\medskip

\section{Graph-directed fractal with integer characteristic}

Suppose $\{K_{i}\}_{i=1}^{p}\ $are dust-like graph-directed sets with ratio $%
r$ and integer characteristic $m.$ As in Subsection \ref{order of block}, we
give rank for every vertex in the corresponding directed graph.

We notice that Theorem 2 follows from Theorem 1 and Lemma \ref{sym}, where $%
r^{k_{1}}=(r^{\prime })^{k_{2}}$ and $m^{k_{1}}=(m^{\prime })^{k_{2}}$ with $%
k_{1},k_{2}\in \mathbb{N}.$

To prove Theorem 1, it suffices to prove the following inductive proposition.

\begin{proposition}
\label{P:k}If the vertex $i$ is of rank $k$ in the directed graph, then $%
K_{i}$ is bilipschitz equivalent to $\Sigma _{m}^{r}.$
\end{proposition}

For this, we need two propositions. The first is on $k=0.$

\begin{proposition}
\label{P:zero}Any set w.r.t. the vertex of rank $0$ is bilipschitz
equivalent to $\Sigma _{m}^{r}.$
\end{proposition}

\begin{proposition}
\label{P:upper}Let $\{K_{1},\cdots ,K_{s},K_{s+1},\cdots K_{s+t}\}$ be
dust-like graph-directed sets with ratio $r,$ integer characteristic $m$ and
block upper triangular adjacent matrix
\begin{equation*}
\left(
\begin{array}{cc}
D_{s\times s} & C_{s\times t} \\
O & B_{t\times t}%
\end{array}%
\right) ,
\end{equation*}%
where $D$ is irreducible and $C\neq 0$. If $K_{s+1},\cdots ,K_{s+t}$ are
bilipschitz equivalent to $\Sigma _{m}^{r}$ respectively, then $K_{i}$ is
also bilipschitz equivalent to $\Sigma _{m}^{r}$ for any $1\leq i\leq s.$
\end{proposition}

Before their proofs, we shall show that

\begin{center}
\textbf{Proposition \ref{P:k} follows from Propositions \ref{P:zero} and \ref%
{P:upper}.}
\end{center}

In fact, inductively we assume that any set w.r.t. the vertex of rank less
than $k$ strictly is bilipschitz equivalent to $\Sigma _{m}^{r}.$ Take a
connected block in $\mathfrak{B}_{k}.$ Without loss of generality, assume
that $1,\cdots ,s$ are all the vertices in this connected block, and $%
s+1,\cdots ,s+t$ all the vertices of rank less than $k$ strictly$.$ By Lemma %
\ref{lemma:order}, we get their adjacent matrix in the form
\begin{equation*}
\left(
\begin{array}{cc}
D & C \\
O & B%
\end{array}%
\right) \text{ where }D\text{ is irreducible and }C\neq 0.
\end{equation*}%
Therefore, Proposition \ref{P:k} follows by induction.

Now, we return to the proofs of Propositions \ref{P:zero} and \ref{P:upper}.

\subsection{Proof of Proposition \protect\ref{P:zero}}

$\ $

We assume that $K_{1},\cdots ,K_{t}$ are dust-like graph-directed set of
ratio $r$ with their adjacent matrix $B$ irreducible such that
\begin{equation*}
Bv=mv\text{ with }v>0.
\end{equation*}%
By Lemma \ref{irreducibe}, we assume that $K_{1},\cdots ,K_{s}$ ($s\leq t)$
are dust-like graph-directed set of ratio $r^{u}$ with their adjacent matrix
$D_{1}>0$ such that
\begin{equation}
D_{1}v^{\ast }=m^{u}v^{\ast }\text{ with }v^{\ast }>0.
\end{equation}%
If $s=1,$ we notice that $K_{1}$ is a disjoint union of $m^{u}$ copies of
itself with ratio $r^{u}.$ We also notice that $\Sigma _{m^{u}}^{r^{u}}$ is
a disjoint union of $m^{u}$ copies of itself with ratio $r^{u},$ and $\Sigma
_{m^{u}}^{r^{u}}$ and $\Sigma _{m}^{r}$ are bilipschitz equivalent. Applying
Lemma \ref{graph} to self-similar sets, we have
\begin{equation*}
K_{1}\text{ and }\Sigma _{m}^{r}\text{ are bilipschitz equivalent.}
\end{equation*}

Now, let $s\geq 2.$ We can assume that
\begin{equation}
u=1.
\end{equation}%
Since $m\in \mathbb{N}$ and the entries of $D_{1}$ are natural integers,
there is an integer eigenvalue w.r.t. the Perron-Frobenius eigenvalue $m$
which is a simple eigenvalue, that means $v^{\ast }$ can be write as
\begin{equation*}
v^{\ast }=(a_{1},\cdots ,a_{s})^{T}\text{ with }a_{1},\cdots ,a_{s}\in
\mathbb{N}.
\end{equation*}%
Without loss of generality, we suppose that
\begin{equation*}
\gcd (a_{1},\cdots ,a_{s})=1.
\end{equation*}%
Here we have%
\begin{equation}
D_{1}(a_{1},\cdots ,a_{s})^{T}=m(a_{1},\cdots ,a_{s})^{T},  \label{Dv=mv}
\end{equation}%
Let $(D_{1})^{k}=(d_{i,j}^{(k)})_{1\leq i,j\leq s}.$ Here $m$ is the
Perron-Frobenius eigenvalue of positive matrix $D_{1},$ that means for all $%
k $ any entry $d_{i,j}^{(k)}\geq C^{-1}m^{k}$ for some constant $C\geq 1.$

Take $k^{\ast }$ such that
\begin{equation*}
C^{-1}m^{k^{\ast }}>2(\max (s,\max_{t}a_{t}))^{2}(\max_{t}a_{t}+\phi
(a_{1},\cdots ,a_{s})).
\end{equation*}
Then we have%
\begin{eqnarray*}
m^{k^{\ast }} &>&2s(\max\nolimits_{t}a_{t})(\max\nolimits_{t}a_{t}+\phi
(a_{1},\cdots ,a_{s})). \\
d_{i,j}^{(k^{\ast })} &>&a_{i}(\max\nolimits_{t}a_{t}+\phi (a_{1},\cdots
,a_{s})).
\end{eqnarray*}%
It follows from (\ref{Dv=mv}) that
\begin{equation}
\sum\nolimits_{j=1}^{s}d_{i,j}^{(k^{\ast })}a_{j}=m^{k^{\ast }}\cdot a_{i}
\label{sum1}
\end{equation}

Let $\Omega _{i}$ be an index set such that $\Omega _{i}=\cup
_{j=1}^{s}\Delta _{i,j}$ is a disjoint union with
\begin{equation*}
\#\Delta _{i,j}=d_{i,j}^{(k^{\ast })}\text{ and }l_{\omega }=a_{j}\text{ for
any }\omega \in \Delta _{i,j}.
\end{equation*}%
Let $\theta (\omega )=j$ if $\omega \in \Delta _{i,j}.$ Then $l_{\omega
}=a_{\theta (\omega )}.$

Applying $b=m^{k^{\ast }},$ $c=a_{i}$ to Lemma \ref{com}, we get the
decomposition
\begin{equation*}
\Omega _{i}=\cup _{t=1}^{a_{i}}\Omega _{(i,t)}
\end{equation*}
such that for all $1\leq t\leq a_{i},$
\begin{equation}
\sum\nolimits_{\omega \in \Omega _{(i,t)}}l_{\omega }=m^{k^{\ast }}\text{.}
\label{d1}
\end{equation}

By the definition of adjacent matrix, for every $i,$ we have a disjoint
union
\begin{equation}
K_{i}=\bigcup_{t=1}^{a_{i}}\bigcup\limits_{\omega \in \Omega
_{(i,t)}}h_{\omega }(K_{\theta (\omega )}),  \label{d2}
\end{equation}%
where $h_{\omega }$ is a similitude with ratio $r^{k^{\ast }}.$

Therefore, $\{K_{1},\cdots ,K_{s}\}$ are dust-like graph-directed sets
satisfying (\ref{d2}).

\medskip

To apply Lemma \ref{com}, we consider the graph-directed sets in the metric
space $\Sigma _{m^{k^{\ast }}}^{r^{k^{\ast }}}$

For integers $\alpha ,\,\beta $ with $1\leq \alpha \leq \beta \leq
m^{k^{\ast }}$, a subset $\Pi _{\alpha }^{\beta }$ of $\Sigma _{m^{k^{\ast
}}}^{r^{k^{\ast }}}$ is defined by
\begin{equation*}
\Pi _{\alpha }^{\beta }=\{x_{1}x_{2}\cdots :x_{1}\in \mathbb{N}\cap \lbrack
\alpha ,\beta ]\}.
\end{equation*}%
Take $\gamma \ast \Pi _{\alpha }^{\beta }=\{x_{1}x_{2}x_{3}\cdots
:x_{1}=\gamma $ and $x_{2}x_{3}\cdots \in \Pi _{\alpha }^{\beta }\}$. For a
word $\pi _{k}=y_{1}\cdots y_{k}$, we give a mapping $(\pi _{k}\ast
):x\mapsto \pi _{k}\ast x=y_{1}\cdots y_{k}x_{1}x_{2}x_{3}\cdots $ for $%
x=x_{1}x_{2}x_{3}\cdots .$ Similarly, let $\pi _{k}\ast B=(\pi _{k}\ast
)(B). $

Then there is a natural similitude from $\Pi _{\alpha }^{\beta }$ to $\gamma
\ast \Pi _{\alpha }^{\beta }$ with ratio $r^{k^{\ast }}$. Notice that
\begin{equation}
\Pi _{1}^{\beta -\alpha +1}\text{and }\Pi _{\alpha }^{\beta }\text{ are
isometric,}  \label{d3}
\end{equation}%
\begin{equation}
\Pi _{1}^{\alpha }=\Pi _{1}^{1}\cup \Pi _{2}^{2}\cup \cdots \Pi _{\alpha
}^{\alpha },  \label{d4}
\end{equation}%
and for a sequence $\{\lambda _{1},\cdots ,\lambda _{t}\}$ with $\lambda
_{1}+\cdots +\lambda _{t}=m^{k^{\ast }}$,
\begin{equation}
\Pi _{1}^{1}=\bigcup_{j=1}^{t}(1\ast \Pi _{\lambda _{1}+\cdots +\lambda
_{j-1}+1}^{\lambda _{1}+\cdots +\lambda _{j}})  \label{d5}
\end{equation}%
where there is a natural similitude from $\Pi _{1}^{\lambda _{j}}$ to $1\ast
\Pi _{\lambda _{1}+\cdots +\lambda _{j-1}+1}^{\lambda _{1}+\cdots +\lambda
_{j}}$ with ratio $r^{k^{\ast }}$.

Give a \emph{total order} \textquotedblleft $<$\textquotedblright\ on the
finite set $\Omega _{(i,t)}$ for every $(i,t),$ for $\omega \in \Omega
_{(i,t)},$ let
\begin{equation*}
\alpha (\omega )=\sum\nolimits_{\omega ^{\prime }<\omega }l_{\omega ^{\prime
}}+1\text{ and }\beta (\omega )=\sum\nolimits_{\omega ^{\prime }<\omega
}l_{\omega ^{\prime }}+l_{\omega }.
\end{equation*}%
Then by (\ref{d3})-(\ref{d5}), we have the decomposition%
\begin{equation}
\Pi _{1}^{a_{i}}=\bigcup_{t=1}^{a_{i}}\Pi
_{t}^{t}=\bigcup_{t=1}^{a_{i}}\bigcup\limits_{\omega \in \Omega
_{(i,t)}}(t\ast \Pi _{\alpha (\omega )}^{\beta (\omega )}),  \label{d6}
\end{equation}

For $\omega \in \Omega _{(i,t)},$ let $\bar{h}_{\omega }$ be a natural
mapping such that%
\begin{equation*}
\bar{h}_{\omega }(\Pi _{1}^{l_{\omega }})=t\ast \Pi _{\alpha (\omega
)}^{\beta (\omega )},
\end{equation*}
then (\ref{d6}) implies%
\begin{equation}
\Pi _{1}^{a_{i}}=\bigcup_{t=1}^{a_{i}}\Pi
_{t}^{t}=\bigcup_{t=1}^{a_{i}}\bigcup\limits_{\omega \in \Omega _{(i,t)}}%
\bar{h}_{\omega }(\Pi _{1}^{l_{\omega }}),  \label{d8}
\end{equation}%
where $\bar{h}_{\omega }$ is a similitude with ratio $r^{k^{\ast }}.$

We obtain that
\begin{equation*}
\{\Pi _{1}^{a_{1}},\Pi _{1}^{a_{2}},\dots ,\Pi _{1}^{a_{s}}\}
\end{equation*}%
are dust-like graph-directed sets on the graph determined by (\ref{d8}) and
each similitude has ratio $r^{k^{\ast }}$.

Notice that $l_{\omega }=a_{\theta (\omega )},$ we have
\begin{equation}
\Pi _{1}^{l_{\omega }}=\Pi _{1}^{a_{\theta (\omega )}}\text{ for }\omega \in
\Omega _{i}.  \label{d9}
\end{equation}%
Then by (\ref{d2}), (\ref{d8})-(\ref{d9}) and Lemma \ref{graph}, $K_{i}$ and
$\Pi _{1}^{a_{i}}$ are bilipschitz equivalent.

Here
\begin{equation*}
\Pi _{1}^{\alpha }=\Pi _{1}^{1}\cup \cdots \cup \Pi _{\alpha }^{\alpha
}=\cup _{t=1}^{\alpha }(t\ast \Sigma _{m^{k^{\ast }}}^{r^{k^{\ast }}})
\end{equation*}
is a disjoint union. By Lemma \ref{copy}, $\Pi _{1}^{\alpha }$ and $\Sigma
_{m^{k^{\ast }}}^{r^{k^{\ast }}}$ are bilipschitz equivalent.$\ $That means $%
K_{i}$ is bilipschitz equivalent to $\Sigma _{m^{k^{\ast }}}^{r^{k^{\ast }}}$%
. Notice that $\Sigma _{m^{k^{\ast }}}^{r^{k^{\ast }}}$ and $\Sigma _{m}^{r}$
are bilipschitz equivalent due to lemma \ref{sym}. Finally, $K_{i}$ is
bilipschitz equivalent to $\Sigma _{m}^{r}$ for every $i$.

This completes the proof of Proposition \ref{P:zero}.

\subsection{Proof of Proposition \protect\ref{P:upper}}

$\ $

Let $L=\left(
\begin{array}{cc}
D_{s\times s} & C_{s\times t} \\
O & B_{t\times t}%
\end{array}%
\right) $ and $L^{k}=(\eta _{i,j}^{(k)})_{i,j}.$

Suppose that $v=(b_{1},\cdots ,b_{s},b_{s+1},\cdots b_{s+t})^{T},$ with $%
b_{i}\in \mathbb{N}$ for all $i,$ is the corresponding eigenvector of
adjacent matrix, i.e.,
\begin{equation*}
Lv=mv.
\end{equation*}

By Lemma \ref{copy}, $\Sigma _{m}^{r}$ and $\Pi _{1}^{b_{i}}$ are
bilipschitz equivalent. For any $1\leq i\leq s,$ it suffices to construct a
bilipschitz bijection from $K_{i}$ to $\Pi _{1}^{b_{i}}.$

Let $v^{\ast }=(b_{1},\cdots ,b_{s})^{T}$ and $v_{1}=(b_{s+1},\cdots
b_{s+t})^{T}>0.$ Since $C\geq 0,C\neq 0$ and $Dv^{\ast }+Cv_{1}=mv^{\ast },$
which means%
\begin{equation*}
Dv^{\ast }\leq mv^{\ast }\text{ and }Dv^{\ast }\neq mv^{\ast }.
\end{equation*}

\begin{lemma}
\label{l:k*}There exists $k^{\ast }$ such that for any $k\geq k^{\ast }$ and
any $i\leq s,$%
\begin{equation*}
\sum\nolimits_{j\leq s}\eta _{i,j}^{(k)}b_{j}<m^{k}.
\end{equation*}
\end{lemma}

\begin{proof}
It suffices to verify that $\rho (D)<m,$ where $\rho (D)$ is the spectrum
radius of $D.$ Otherwise, we suppose that $\rho (D)\geq m$.

For irreducible matrix $D$, by Lemma \ref{spec}, we have
\begin{equation*}
m\leq \rho (D)\leq \max_{i}\frac{(Dv^{\ast })_{i}}{(v^{\ast })_{i}}\leq m,
\end{equation*}%
which implies $\rho (D)=\max_{i}\frac{(Dv^{\ast })_{i}}{(v^{\ast })_{i}}=m.$
By using Lemma \ref{spec}, we have
\begin{equation*}
Dv^{\ast }=mv^{\ast }.
\end{equation*}%
This contradicts to $Dv^{\ast }\neq mv^{\ast }$.
\end{proof}

Without loss of generality, we assume that for any $j,$
\begin{equation}
b_{j}<m\text{ and }k^{\ast }=1  \label{b<m}
\end{equation}%
in Lemma \ref{l:k*}. In fact, otherwise we can replace $m$ by $m^{k},$ then $%
L^{k}v=m^{k}v$ and $\Sigma _{m}^{r}$ and $\Sigma _{m^{k}}^{r^{k}}$ are
bilipschitz equivalent.

Now, for any $i\leq s,$
\begin{equation}
d_{i}=\sum\nolimits_{j\leq s}\eta _{i,j}b_{j}<m.  \label{spectrum}
\end{equation}%
Notice that
\begin{equation*}
\sum\nolimits_{1\leq j\leq s+t}\eta _{i,j}b_{j}=m\cdot b_{i},
\end{equation*}%
where $mb_{i}<m^{2}$ due to (\ref{b<m}).

Let $\delta _{i}=\sum\nolimits_{1\leq j\leq s+t}\eta _{i,j}$. For each $i,$
let $\{c(i,q)\}_{q=1}^{\delta _{i}}$ denote the following sequence
\begin{equation*}
\overset{\eta _{i,1}}{\overbrace{b_{1},\cdots ,b_{1}},}\overset{\eta _{i,2}}{%
\overbrace{b_{2},\cdots ,b_{2}},}\cdots \cdots \overset{\eta _{i,s}}{,%
\overbrace{b_{s},\cdots ,b_{s}},}\overset{\eta _{i,s+1}}{\overbrace{%
b_{s+1},\cdots ,b_{s+1}},}\cdots ,\overset{\eta _{i,s+t}}{\overbrace{%
b_{s+t},\cdots ,b_{s+t}}},
\end{equation*}%
and
\begin{equation*}
\gamma _{i,q}=j_{0}\text{ if and only if }\sum\nolimits_{1\leq j\leq
j_{0}-1}\eta _{i,j}<q\leq \sum\nolimits_{1\leq j\leq j_{0}}\eta _{i,j}.
\end{equation*}%
Then
\begin{equation*}
c(i,q)=b_{\gamma _{i,q}}.
\end{equation*}%
Set
\begin{eqnarray*}
\zeta (i,q) &=&\sum\nolimits_{j<q}c(i,j)+1, \\
\kappa (i,q) &=&\sum\nolimits_{j\leq
q}c(i,j)=\sum\nolimits_{j<q}c(i,j)+c(i,q).
\end{eqnarray*}%
\begin{equation*}
\overset{c(i,1)}{\overbrace{\underset{\underset{1}{\downarrow }}{1},\cdots ,%
\underset{\underset{c(i,1)}{\downarrow }}{1}}}\overset{c(i,2)}{\ \overbrace{%
1,\cdots ,1}}\cdots \cdots \overset{c(i,q)}{,\overbrace{\underset{\underset{%
\zeta (i,q)}{\downarrow }}{1},\cdots ,\underset{\underset{\kappa (i,q)}{%
\downarrow }}{1}},}\cdots \cdots
\end{equation*}%
We also have
\begin{equation}
\zeta (i,q)=\sum\nolimits_{1\leq j\leq \gamma _{i,q}-1}(\eta
_{i,j}b_{j})+(q-\sum\nolimits_{1\leq j\leq \gamma _{i,q}-1}\eta
_{i,j}-1)b_{\gamma _{i,q}}+1.  \label{spectrum2}
\end{equation}%
By (\ref{spectrum}), we have
\begin{equation}
\kappa (i,q)<m\text{ if }\gamma _{i,q}\leq s.  \label{spectrum3}
\end{equation}

Then we have the decomposition of $K_{i}$ for $1\leq i\leq s,$%
\begin{equation*}
K_{i}=\cup _{q=1}^{\delta _{i}}S_{i,q}(K_{\gamma _{i,q}}),
\end{equation*}%
where $S_{i,q}$ is a similitude of ratio $r$ for each $q.$

Given a point $x\in K_{i_{1}},$ we get its \emph{code, }an infinite sequence%
\begin{equation*}
(i_{1},q_{1})(i_{2},q_{2})(i_{3},q_{3})\cdots \text{ with }\gamma
_{i_{k},q_{k}}=i_{k+1},\text{ }q_{k}\leq \delta _{i_{k}}\text{ for all }k,
\end{equation*}%
that is $x\in S_{i_{1},q_{1}}(K_{i_{2}}),x\in
S_{i_{1},q_{1}}S_{i_{2},q_{2}}(K_{i_{3}}),x\in
S_{i_{1},q_{1}}S_{i_{2},q_{2}}S_{i_{3},q_{3}}(K_{i_{4}})\cdots ,$ i.e.,%
\begin{equation*}
\{x\}=\bigcap\nolimits_{k}S_{i_{1},q_{1}}\cdots S_{i_{k},q_{k}}(K_{i_{k+1}}).
\end{equation*}

Let $[(i_{1},q_{1})\cdots (i_{k},q_{k})]$ denote the code cylinder
\begin{equation*}
\{x:x\text{ has code }(j_{1},p_{1})\cdots (j_{k},p_{k})\cdots \text{ with }%
j_{u}=i_{u},p_{u}=q_{u}\text{ for any }u\leq k\}.
\end{equation*}

\bigskip

We will construct the bilipschitz bijection by induction. Let
\begin{equation*}
\mathfrak{C}_{(i-1)m+j}=i\ast j\ast \Sigma _{m}^{r}\text{ for }1\leq i,j\leq
m.
\end{equation*}%
Then
\begin{equation*}
\Pi _{1}^{b_{i}}=\cup _{i=1}^{b_{i}}\cup _{j=1}^{m}\mathfrak{C}_{(i-1)m+j},
\end{equation*}%
where $b_{i}<m$ due to (\ref{b<m}).

Since $K_{s+1},\cdots ,K_{s+t}$ are bilipschitz equivalent to $\Sigma
_{m}^{r},$ by Lemma \ref{copy}, for any $(i,q,j)$ with $q>d_{i}$ we have a
natural bilipschitz bijection
\begin{equation*}
h_{i,q,j}:K_{\gamma _{i,q}}\rightarrow \bigcup\nolimits_{k=\zeta
(i,q)}^{\kappa (i,q)}\mathfrak{C}_{k+(j-1)m}.
\end{equation*}

Define $f_{i_{1}}$ on the $[(i_{1},q_{1})]$ with $i_{2}=\gamma
_{i_{1},q_{1}}>s$ such that
\begin{equation*}
f_{i_{1}}|_{[(i_{1},q_{1})]}=h_{i_{1},q_{1},1}\circ S_{i_{1},q_{1}}^{-1}%
\text{ for }i_{2}=\gamma _{i_{1},q_{1}}>s.
\end{equation*}%
Then we have
\begin{equation}
f_{i_{1}}([(i_{1},q_{1})])=\bigcup\nolimits_{j=\zeta (i_{1},q_{1})}^{\kappa
(i_{1},q_{1})}\mathfrak{C}_{j}\text{ for }i_{2}=\gamma _{i_{1},q_{1}}>s.
\label{image0}
\end{equation}

For $i_{2}=\gamma _{i_{1},q_{1}}\leq s$, as in (\ref{image0}), we give a
mapping $\mathfrak{F}_{i_{1}},$ which maps a code cylinders to a union of
finitely many cylinders in $\Sigma _{m}^{r},$ such that
\begin{equation}
\mathfrak{F}_{i_{1}}([(i_{1},q_{1})])=\bigcup\nolimits_{j=\zeta
(i_{1},q_{1})}^{\kappa (i_{1},q_{1})}\mathfrak{C}_{j}\text{ for }%
i_{2}=\gamma _{i_{1},q_{1}}\leq s.  \label{image}
\end{equation}%
When $i_{2}=\gamma _{i_{1},q_{1}}\leq s,$ by (\ref{spectrum}) and (\ref%
{spectrum3}), we have
\begin{equation*}
\kappa (i_{1},q_{1})<m,
\end{equation*}%
then (\ref{image}) means
\begin{equation}
\mathfrak{F}_{i_{1}}([(i_{1},q_{1})])=1\ast \bigcup\nolimits_{j=\zeta
(i_{1},q_{1})}^{\kappa (i_{1},q_{1})}\Pi _{j}^{j}\text{ }\ \text{if }%
i_{1},i_{2}\leq s.  \label{image2}
\end{equation}

As in Example \ref{e:p4}, for
\begin{equation*}
K_{i_{1}}=\bigcup\nolimits_{q_{1}}S_{i_{1},q_{1}}(K_{\gamma _{i_{1},q_{1}}})
\end{equation*}
we arrange its small parts $\{S_{i_{1},q_{1}}(K_{\gamma
_{i_{1},q_{1}}})\}_{q_{1}}$ from left to right, in the same time, we arrange
the small parts $\{\bigcup\nolimits_{j=\zeta (i_{1},q_{1})}^{\kappa
(i_{1},q_{1})}\mathfrak{C}_{j}\}_{q_{1}}$ \ of
\begin{equation*}
\Pi _{1}^{b_{i_{1}}}=\bigcup\nolimits_{q_{1}}\bigcup\nolimits_{j=\zeta
(i_{1},q_{1})}^{\kappa (i_{1},q_{1})}\mathfrak{C}_{j}
\end{equation*}
from left to right. Then $\mathfrak{F}_{i_{1}}$ (when $i_{2}\leq s$) and $%
f_{i_{1}}$ (when $i_{2}>s$) corresponds $S_{i_{1},q_{1}}(K_{\gamma
_{i_{1},q_{1}}})=[(i_{1},q_{1})]$ to $\bigcup\nolimits_{j=\zeta
(i_{1},q_{1})}^{\kappa (i_{1},q_{1})}\mathfrak{C}_{j}.$ Notice that $%
\bigcup\nolimits_{j=\zeta (i_{1},q_{1})}^{\kappa (i_{1},q_{1})}\mathfrak{C}%
_{j}$ includes
\begin{equation*}
\kappa (i_{1},q_{1})-\zeta (i_{1},q_{1})=c(i_{1},q_{1})=b_{\gamma
_{i_{1},q_{1}}}
\end{equation*}%
smaller cylinders, where $\sum\nolimits_{q_{1}}c(i_{1},q_{1})=m\cdot
b_{i_{1}}.$

\begin{example}
\label{e:p4}Let $m=4,$ $s=t=2,$ the matrix w.r.t. the dust-like
graph-directed sets is
\begin{equation*}
M=\left(
\begin{array}{cccc}
0 & 1 & 1 & 0 \\
1 & 1 & 1 & 1 \\
0 & 0 & 1 & 2 \\
0 & 0 & 3 & 2%
\end{array}%
\right) .
\end{equation*}%
Take $v=(b_{1},b_{2},b_{3},b_{4})^{T}=(1,2,2,3)^{T}$ with $Mv=4v.$ Then we
have the first step of $f_{i}$ and $\mathfrak{F}_{i}$ $(i=1,2)$ as in the
following figure.
\end{example}

\vspace{-0.5cm} 
\begin{figure}[tbph]
\centering\includegraphics[width=1\textwidth]{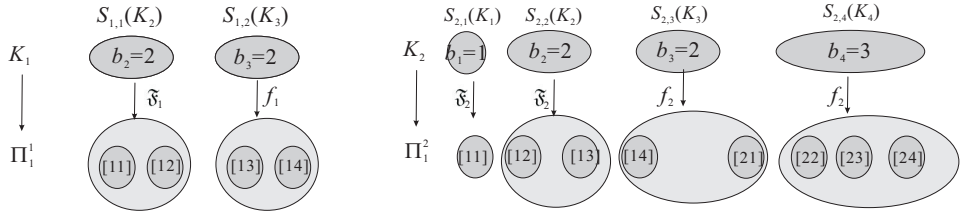} \vspace{-0.3cm}
\caption{The first step of $f_{i}$ and $\mathfrak{F}_{i}$ $(i=1,2)$ }
\end{figure}
\vspace{-0.2cm}

We will define $\mathfrak{F}_{i_{1}}$ and $f_{i_{1}}$ by induction.

Given cylinder $[(i_{1},q_{1})\cdots (i_{k},q_{k})],$ we denote $%
g_{i_{1}}=\left\{
\begin{array}{c}
f_{i_{1}}\text{ if }i_{k}>s, \\
\mathfrak{F}_{i_{1}}\text{ if }i_{k}\leq s.%
\end{array}%
\right. $

We assume for cylinder $[(i_{1},q_{1})\cdots (i_{k},q_{k})]$ with $%
i_{1},i_{2},\cdots ,i_{k-1}\leq s$,
\begin{equation}
g_{i_{1}}([(i_{1},q_{1})\cdots (i_{k},q_{k})])=\pi _{k-1}\ast
\bigcup\nolimits_{j=\zeta (i_{k},q_{k})}^{\kappa (i_{k},q_{k})}\mathfrak{C}%
_{j+m(\zeta (i_{k-1},q_{k-1})-1)}.  \label{ind0}
\end{equation}%
where $\pi _{k-1}$ is a word of length $k-1,$ in particular, for $i_{k}\leq
s,$%
\begin{equation}
\mathfrak{F}_{i_{1}}([(i_{1},q_{1})\cdots (i_{k},q_{k})])=\pi _{k-1}\ast
\zeta (i_{k-1},q_{k-1}))\ast \bigcup\nolimits_{j=\zeta
(i_{k},q_{k})}^{\kappa (i_{k},q_{k})}\Pi _{j}^{j},  \label{ind2}
\end{equation}%
and for $i_{k}>s,$
\begin{equation}
f_{i_{1}}|_{[(i_{1},q_{1})\cdots (i_{k},q_{k})]}=\pi _{k-1}\ast
h_{i_{k},q_{k},\zeta (i_{k-1},q_{k-1})}\circ (S_{i_{1},q_{1}}\circ \cdots
\circ S_{i_{k},q_{k}})^{-1}.  \label{ind-1}
\end{equation}

\begin{figure}[tbph]
\centering\includegraphics[width=1\textwidth]{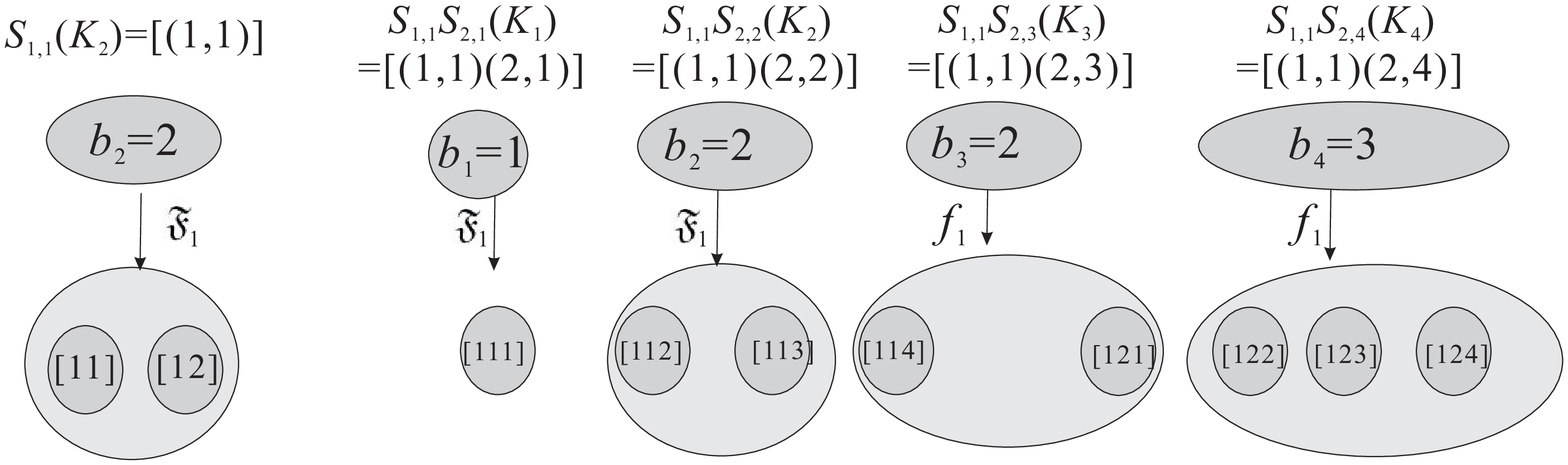} \vspace{-0.3cm}
\caption{The second step of $\mathfrak{F}_{1}$ on $[(1,1)]$ in Example 6}
\end{figure}
\begin{figure}[tbph]
\centering\includegraphics[width=1\textwidth]{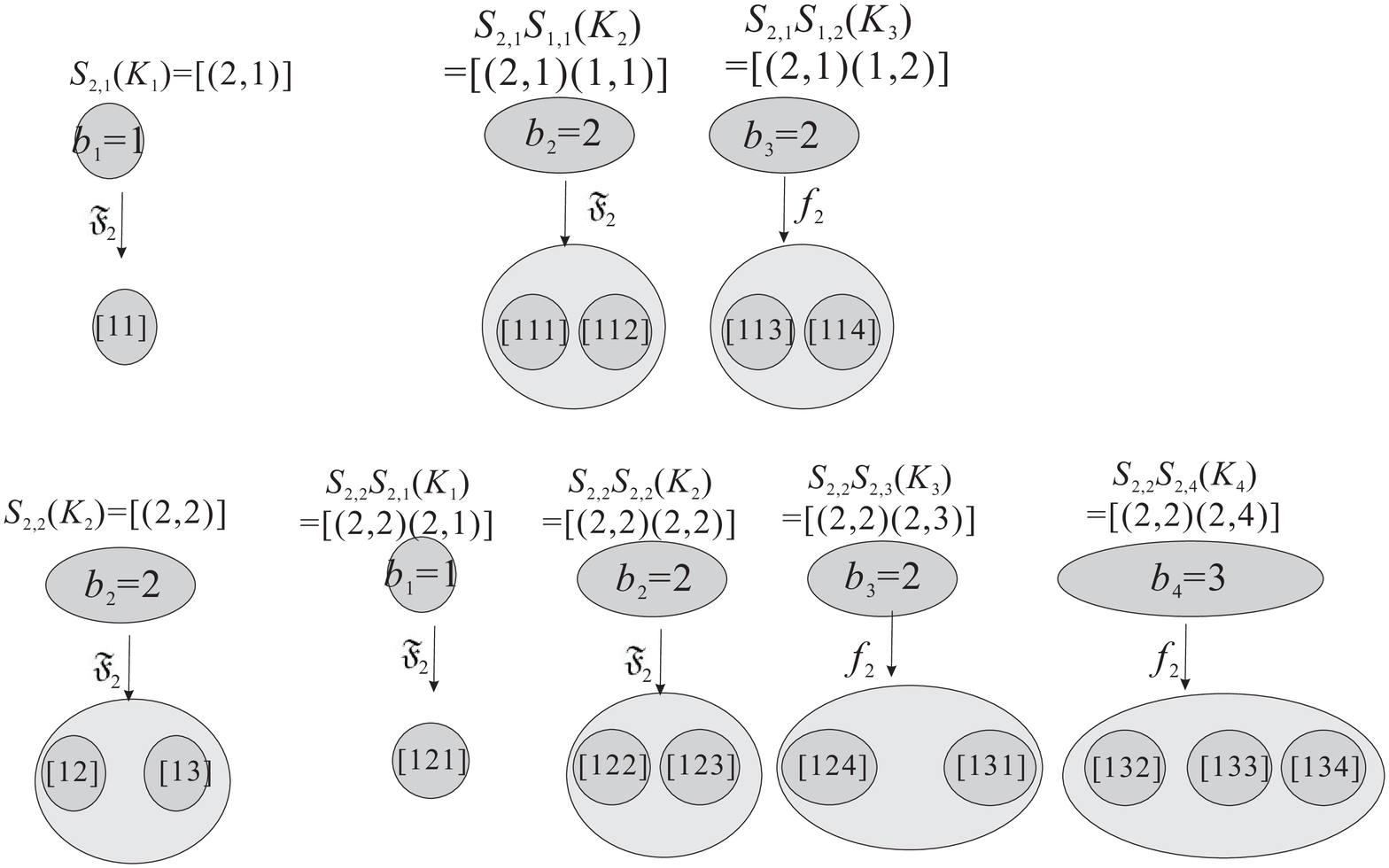} \vspace{-0.3cm}
\caption{The second step of $\mathfrak{F}_{2}$ on $[(2,1)], [(2,2)]$ in
Example 6}
\end{figure}

Inductively, for $[(i_{1},q_{1})\cdots (i_{k},q_{k})(i_{k+1},q_{k+1})]$ with
\begin{equation*}
i_{1},i_{2},\cdots ,i_{k}\leq s,
\end{equation*}%
we have $\zeta (i_{k-1},q_{k-1})<m.$ Let
\begin{equation}
\pi _{k}=\pi _{k-1}\ast \zeta (i_{k-1},q_{k-1})  \label{ind4}
\end{equation}%
be a word of length $k.$

\textbf{Case A}: For $i_{k+1}\leq s,$ we have $\zeta (i_{k},q_{k})\leq
\kappa (i_{k},q_{k})<m$ and let%
\begin{equation}
\mathfrak{F}_{i_{1}}([(i_{1},q_{1})\cdots (i_{k+1},q_{k+1})])=\pi _{k}\ast
\zeta (i_{k},q_{k})\ast \bigcup\nolimits_{j=\zeta (i_{k},q_{k})}^{\kappa
(i_{k},q_{k})}\Pi _{j}^{j}.  \label{ind6}
\end{equation}

\textbf{Case B}: For $i_{k+1}=\gamma _{i_{k},q_{k}}>s,$ we let
\begin{eqnarray}
&&f_{i_{1}}|_{[(i_{1},q_{1})\cdots (i_{k},q_{k})(i_{k+1},q_{k+1})]}
\label{ind-2} \\
&=&\pi _{k}\ast h_{i_{k+1},q_{k+1},\zeta (i_{k},q_{k})}\circ
(S_{i_{1},q_{1}}\circ \cdots \circ S_{i_{k+1},q_{k+1}})^{-1}.  \notag
\end{eqnarray}%
Therefore, We have
\begin{eqnarray}
&&g_{i_{1}}([(i_{1},q_{1})\cdots (i_{k},q_{k})(i_{k+1},q_{k+1})])
\label{ind5} \\
&=&\pi _{k}\ast \bigcup\nolimits_{j=\zeta (i_{k+1},q_{k+1})}^{\kappa
(i_{k+1},q_{k+1})}\mathfrak{C}_{j+m(\zeta (i_{k},q_{k})-1)}.  \notag
\end{eqnarray}%
where $g_{i_{1}}=\left\{
\begin{array}{c}
f_{i_{1}}\text{ if }i_{k+1}>s, \\
\mathfrak{F}_{i_{1}}\text{ if }i_{k+1}\leq s.%
\end{array}%
\right. $

Then $f_{i_{1}}$ and $\mathfrak{F}_{i_{1}}$ have been defined inductively.

\bigskip

For any $i\leq s,$ the following set is dense in $K_{i}:$
\begin{equation*}
H_{i}=\{x:x\text{ }\text{has code }(i_{1},q_{1})\cdots (i_{u},q_{u})\cdots
\text{with }i_{1}=i\text{ and }i_{u}>s\text{ for some }u\}.
\end{equation*}%
By induction $f_{i}$ has been defined on $H_{i}$ for every $i\leq s.$

Therefore, there is a natural extension from $K_{i}\ $to $\Pi _{1}^{b_{i}}$,
still denoted by $f_{i}.$ It follows from (\ref{ind0}) that for $%
i_{1},i_{2},\cdots ,i_{k-1}\leq s,$
\begin{equation}
f_{i_{1}}([(i_{1},q_{1})\cdots (i_{k},q_{k})])=\pi _{k-1}\ast
\bigcup\nolimits_{j=\zeta (i_{k},q_{k})}^{\kappa (i_{k},q_{k})}\mathfrak{C}%
_{j+m(\zeta (i_{k-1},q_{k-1})-1)}.  \label{imp}
\end{equation}

We shall prove that $f_{i}:K_{i}\rightarrow \Pi _{1}^{b_{i}}$ is a
bilipschitz bijection.

In fact, given $i_{1}\leq s,$ suppose distinct points $x,y\in K_{i_{1}}$
with codes
\begin{equation*}
(i_{1},q_{1})\cdots (i_{u-1},q_{u-1})(i_{u},q_{u})\cdots \text{ and }%
(i_{1},q_{1})\cdots (i_{u-1},q_{u-1})(i_{u},p_{u})\cdots
\end{equation*}%
respectively, where $p_{u}\neq q_{u}$.

There are two cases as follows. \newline
\textbf{Case 1}:\textbf{\ }If $i_{k}>s$ for some $k<u,$ we assume that $%
i_{1},\cdots ,i_{k-1}\leq s.$

Then $x,y\in \lbrack (i_{1},q_{1})\cdots (i_{k},q_{k})]$ and by (\ref{ind-1}%
) there is a word $\pi _{k-1}$\ of length $k-1$ such that
\begin{equation*}
f_{i_{1}}|_{[(i_{1},q_{1})\cdots (i_{k},q_{k})]}=\pi _{k-1}\ast
h_{i_{k},q_{k},\zeta (i_{k-1},q_{k-1})}\circ (S_{i_{1},q_{1}}\circ \cdots
\circ S_{i_{k},q_{k}})^{-1}
\end{equation*}%
is bilipschitz mapping, with the bilipschitz constant
\begin{eqnarray*}
blip(f_{i_{1}}|_{[(i_{1},q_{1})\cdots (i_{k},q_{k})]}) &\leq &r^{-1}\cdot
blip(h_{i_{k},q_{k},\zeta (i_{k-1},q_{k-1})}) \\
&\leq &r^{-1}\max_{i,q,j}blip(h_{i,q,j}),
\end{eqnarray*}%
since the mappings $(\pi _{k-1}\ast )$ and $(S_{i_{1},q_{1}}\circ \cdots
\circ S_{i_{k},q_{k}})^{-1}$ are similitudes with ratios $r^{k-1}$ and $%
r^{-k}$ respectively. Therefore, we have%
\begin{equation*}
\frac{r}{\max\limits_{i,q,j}blip(h_{i,q,j})}\leq \frac{\text{d}_{\Sigma
_{m}^{r}}(f_{i_{1}}(x),f_{i_{1}}(y))}{\text{d}(x,y)}\leq \frac{%
\max\limits_{i,q,j}blip(h_{i,q,j})}{r}.
\end{equation*}%
\medskip \newline
\textbf{Case 2}: If $i_{1},\cdots ,i_{u-1}\leq s$, then by (\ref{imp}), we
have%
\begin{eqnarray*}
&&f_{i_{1}}([(i_{1},q_{1})\cdots (i_{u-1},q_{u-1})(i_{u},q_{u})])=\pi
_{u-1}\ast \bigcup\nolimits_{j=\zeta (i_{u},q_{u})}^{\kappa (i_{u},q_{u})}%
\mathfrak{C}_{j+m(\zeta (i_{u-1},q_{u-1})-1)}, \\
&&f_{i_{1}}([(i_{1},q_{1})\cdots (i_{u-1},q_{u-1})(i_{u},p_{u})])=\pi
_{u-1}\ast \bigcup\nolimits_{j^{\prime }=\zeta (i_{u},p_{u})}^{\kappa
(i_{u},p_{u})}\mathfrak{C}_{j^{\prime }+m(\zeta (i_{u-1},q_{u-1})-1)},
\end{eqnarray*}%
where
\begin{equation*}
\mathfrak{C}_{j+m(\zeta (i_{u-1},q_{u-1})-1)}\cap \mathfrak{C}_{j^{\prime
}+m(\zeta (i_{u-1},q_{u-1})-1)}=\emptyset ,
\end{equation*}%
for any $\zeta (i_{u},q_{u})\leq j\leq \kappa (i_{u},q_{u})$ and $\zeta
(i_{u},p_{u})\leq j^{\prime }\leq \kappa (i_{u},p_{u}).$ Hence%
\begin{equation*}
r^{u+1}\leq \text{d}_{\Sigma _{m}^{r}}(f_{i_{1}}(x),f_{i_{1}}(y))\leq
r^{u-1}.
\end{equation*}%
In the same time, $x,y\in S_{i_{1},q_{1}}S_{i_{2},q_{2}}\cdots
S_{i_{u-1},q_{u-1}}(K_{i_{u}})$, we notice that
\begin{equation*}
\text{d}(x,y)\leq r^{u-1}|K_{i_{u}}|\leq r^{u-1}\max_{t}|K_{t}|.
\end{equation*}%
Since $S_{i_{1},q_{1}}\circ S_{i_{2},q_{2}}\circ \cdots \circ
S_{i_{u-1},q_{u-1}}$ has ratio $r^{u-1},$ we have%
\begin{eqnarray*}
\text{d}(x,y) &\geq &r^{u-1}\text{d}(S_{i_{u},q_{u}}(K_{\gamma
_{i_{u,}q_{u}}}),S_{i_{u},p_{u}}(K_{\gamma _{i_{u,}p_{u}}})) \\
&\geq &r^{u-1}\min_{i,q\neq q^{\prime }}\text{d}(S_{i,q}(K_{\gamma
_{i,q}}),S_{i,q^{\prime }}(K_{\gamma _{i,q^{\prime }}})).
\end{eqnarray*}%
Therefore, in this case we have%
\begin{equation*}
\frac{r^{2}}{\max_{t}|K_{t}|}\leq \frac{\text{d}_{\Sigma
_{m}^{r}}(f_{i_{1}}(x),f_{i_{2}}(y))}{\text{d}(x,y)}\leq \frac{1}{%
\min_{i,q\neq q^{\prime }}\text{d}(S_{i,q}(K_{\gamma _{i,q}}),S_{i,q^{\prime
}}(K_{\gamma _{i,q^{\prime }}}))}.
\end{equation*}%
Then $f_{i}$ is bilipschitz bijection from $K_{i}$ to $\Pi _{1}^{b_{i}}$ for
any $i\leq s.$

\medskip

Then Proposition 4 is proved.

\bigskip

\section{Self-similar set}

In this section, we will obtain the graph-directed fractals due to total
disconnectedness, and verify their integer characteristic due to
non-existence of complete overlaps, then we can prove Theorem 3 by using
Theorem 1. Finally we prove Proposition 1, which is the basis of efficient
algorithm to test the existence of complete overlaps.

\subsection{Total disconnectedness}

$\ $

Let $E$ be a totally disconnected self-similar set generated by%
\begin{equation*}
S_{i}(x)=g_{i}(x)/n+b_{i},\text{ }i=1,\cdots ,m,
\end{equation*}%
where isometry $g_{i}$ satisfies $|g_{i}(x)|=|x|$ for any $x$.

Let
\begin{equation*}
\mathbb{H}=\cup _{g\in \mathbb{G}}\cup _{b\in \Gamma }(gE+nb)
\end{equation*}%
Then for any closed ball $B(0,\beta ),$ set
\begin{equation*}
\digamma _{\beta }=\{(g,b)\in \mathbb{G}\times \Gamma :(gE+nb)\cap B(0,\beta
)\neq \emptyset \}.
\end{equation*}%
We have
\begin{equation*}
\#\digamma _{\beta }<\infty
\end{equation*}%
Then it follows from Lemma \ref{l:total}\ that
\begin{equation*}
\mathbb{H}\cap B(0,\beta )=\bigcup\nolimits_{(g,b)\in \digamma _{\beta
}}(gE+nb)\cap B(0,\beta )
\end{equation*}%
is totally disconnected.

We notice that $E\subset B(0,\delta )$ the closed ball with center $0$ and
radius $\delta =\frac{\max_{i}|b_{i}|}{1-1/n},$ since isometry $%
|g_{i}(x)|=|x|$ and
\begin{equation*}
|x|\leq \max_{i}|b_{i}|(1+\frac{1}{n}+\frac{1}{n^{2}}+\cdots )\leq \frac{%
\max_{i}|b_{i}|}{1-1/n}\text{ for any }x\in E.
\end{equation*}%
Let $\mathbb{H}_{k}=\cup _{g\in \mathbb{G}}\cup _{b\in \Gamma }(gE_{k}+nb),$
where
\begin{equation*}
E_{k}=\bigcup\nolimits_{i_{1}\cdots i_{k}}(S_{i_{1}\cdots i_{k}}(0)+\frac{%
B(0,\delta )}{n^{k}}).
\end{equation*}%
Here $E_{k}\overset{d_{H}}{\rightarrow }E$ and thus $\mathbb{H}_{k}\overset{%
d_{H}}{\rightarrow }\mathbb{H}.$

\begin{lemma}
\label{l:there}There exist a constant $\alpha $ and an integer $k^{\ast }$
such that $E\subset B(0,\alpha )$ and any connected components in $\mathbb{H}%
_{k^{\ast }}$ touching $B(0,2\delta )$ cannot touch $\{x:|x|\geq 2\delta
+1\}.$
\end{lemma}

\begin{proof}
In fact, we assume on the contrary, for each $k$ there is a connected
components $X_{k}$ in $E_{k}$ touching $B(0,2\delta )$ and $\{x:|x|\geq
2\delta +1\}$. Therefore, for every $k$ we can take a curve $\gamma
_{k}\subset E_{k}\cap \{x:\delta \leq |x|\leq 2\delta +1\}$ with its
endpoints lying in $\{x:|x|=2\delta \}$ and $\{x:|x|=2\delta +1\}$
respectively. By Lemma \ref{l:connected} and $\mathbb{H}_{k}\overset{d_{H}}{%
\rightarrow }\mathbb{H}$, for some subsequence $\{k_{i}\}_{i},$ we have $%
\gamma _{k_{i}}\overset{d_{H}}{\rightarrow }$ $X$, where $X\subset
B(0,2\delta +1)$ is a connected set in $\mathbb{H}\cap B(0,2\delta +1)$
touching $B(0,2\delta )$ and $\{x:|x|\geq 2\delta +1\}.$ This contradicts
the fact that $\mathbb{H}\cap B(0,2\delta +1)$ is totally disconnected.
\end{proof}

We notice that $E$ is also generated by $\{S_{i_{1}\cdots i_{k^{\ast
}}}\}_{i_{1}\cdots i_{k^{\ast }}}$ with ratio $n^{-k^{\ast }}$ and $%
S_{i_{1}\cdots i_{k^{\ast }}}(0)\in \Gamma /n^{(k^{\ast }-1)}.$ Let
\begin{equation*}
\Gamma ^{\prime }=\Gamma /n^{k^{\ast }-1}
\end{equation*}%
be the new discrete additional group of $\mathbb{R}^{l}$ and $n^{\prime
}=n^{k^{\ast }}$ and $\delta ^{\prime }=\delta /n^{k^{\ast }-1}.$ Here $%
g:\Gamma ^{\prime }\rightarrow \Gamma ^{\prime }$ for any $g\in \mathbb{G}$
since $g$ is linear in $\mathbb{R}^{l}.$ Let $E$ be generated by
\begin{equation*}
S_{i}^{\prime }(x)=\frac{g_{i}^{\prime }(x)}{n^{\prime }}+b_{i}^{\prime }%
\text{ where }i=1,\cdots ,m^{\prime }\text{ with }n^{\prime }=n^{k^{\ast
}},m^{\prime }=m^{k^{\ast }}.
\end{equation*}

\begin{definition}
We call a finite set $D\subset \mathbb{G}\times \Gamma ^{\prime }$ a type,
if
\begin{equation*}
\chi _{D}=\bigcup_{(g,d)\in D}(\frac{B(0,\delta )}{n^{\prime }}+d)\text{ is
connected}.
\end{equation*}%
We say $D_{1},D_{2}$ are equivalent, denoted by $D_{1}\sim D_{2}$, if there
exist $g\in \mathbb{G}$ and $b\in n\Gamma $ such that $\lambda
_{g,b}(D_{2})=D_{1},$where
\begin{equation*}
\lambda _{g,b}(g_{2},d_{2})=(gg_{2},gd_{2}+b)
\end{equation*}%
for $(g_{2},d_{2})\in D_{2}.$ Denote $[D]$ the equivalence class w.r.t. $D.$
\end{definition}

If $D$ is a type, we consider the compact set
\begin{equation}
K_{D}=\bigcup_{(g,d)\in D}(\frac{gE}{n^{\prime }}+d).  \label{generate}
\end{equation}%
In particular,
\begin{equation*}
E=K_{\{(g_{1}^{\prime },b_{1}^{\prime }),\cdots ,\{g_{m^{\prime }}^{\prime
},b_{m^{\prime }}^{\prime }\}\}}.
\end{equation*}

\begin{lemma}
$K_{D}\subset \chi _{D.}$
\end{lemma}

\begin{proof}
It suffice to notice that $E\subset B(0,\delta )$ and $g(B(0,\delta
))=B(0,\delta )$ for $g\in \mathbb{G}.$
\end{proof}

\begin{lemma}
\label{iso}If $D_{1}\sim D_{2}$, then $K_{D_{1}}=S(K_{D_{2}})$ for some
isometry $S.$
\end{lemma}

\begin{proof}
Suppose $\lambda _{(g,b)}(D_{2})=D_{1}$ for some $g\in \mathbb{G}$ and $b\in
n\Gamma .$ Here%
\begin{equation*}
\lambda _{(g,b)}(g_{2},d_{2})=(gg_{2},gd_{2}+b),
\end{equation*}%
which implies%
\begin{eqnarray*}
K_{D_{1}} &=&\bigcup_{(g_{1},d_{1})\in D_{1}}(\frac{g_{1}E}{n^{\prime }}%
+d_{1}) \\
&=&\bigcup_{(g_{2},d_{2})\in D_{2}}(\frac{gg_{2}E}{n^{\prime }}+gd_{2}+b) \\
&=&g\left( \bigcup_{(g_{2},d_{2})\in D_{2}}(\frac{g_{2}E}{n^{\prime }}%
+d_{2})\right) +b \\
&=&g(K_{D_{2}})+b.
\end{eqnarray*}%
This means $K_{D_{1}}=S(K_{D_{2}})$ for some isometry $S.$
\end{proof}

We have the following decomposition for (\ref{generate}):%
\begin{equation*}
K_{D}=\frac{1}{n^{\prime }}\left( \bigcup_{(g,d)\in
D}\bigcup_{i=1}^{m^{\prime }}(\frac{1}{n^{\prime }}(gg_{i}^{\prime
}E)+gb_{i}^{\prime }+n^{\prime }d)\right) .
\end{equation*}%
We give a type decomposition $\cup _{i}D_{i}$ of $\{(gg_{i}^{\prime
},gb_{i}^{\prime }+n^{\prime }d)\}_{(g,d)\in D,\text{ }i\leq m^{\prime }}$
such that $\{D_{i}\}_{i}$ are types and%
\begin{equation}
\chi _{D_{i}}\cap \chi _{D_{j}}=\emptyset \text{ for }i\neq j,  \label{e:de}
\end{equation}%
i.e., $\{\chi _{D_{i}}\}_{i}$ are connected components of $\bigcup_{(g,d)\in
D}\bigcup_{i=1}^{m^{\prime }}(\frac{B(0,\delta )}{n^{\prime }}%
+gb_{i}^{\prime }+n^{\prime }d).$

Therefore, we get a decomposition
\begin{equation}
K_{D}=\frac{1}{n^{\prime }}\bigcup_{D^{\prime }}K_{D^{\prime }}
\label{e:decomp}
\end{equation}%
which is a \emph{disjoint} union due to the $K_{D}\subset \chi _{D.}$

Then we say that $[D^{\prime }]$ above is generated by $[D]\ $\emph{directly}%
, denoted by $[D]\rightarrow \lbrack D^{\prime }].$ We also say that $%
[D_{k}] $ is generated by $[D_{1}],$ if
\begin{equation*}
\lbrack D_{1}]\rightarrow \lbrack D_{2}]\rightarrow \cdots \rightarrow
\lbrack D_{k}].
\end{equation*}

Let $\mathbb{G}(D_{1})=\cup _{g\in \mathbb{G}}\cup _{i=1}^{m^{\prime
}}(gg_{i}^{\prime },gb_{i}^{\prime }).$

Suppose $[(g^{\prime },d^{\prime })]$ is generated by $[(g,d)]$ directly,
then it follows from (\ref{e:decomp}) that%
\begin{equation*}
(g^{\prime },d^{\prime })\sim (gg_{i}^{\prime },gb_{i}^{\prime }+n^{\prime
}d),\text{ for some }i\leq m^{\prime }\text{ and }d\in \Gamma ^{\prime }.
\end{equation*}%
Then by the definition of equivalence, there exist $g_{1}\in \mathbb{G}$ and
$b_{1}\in n\Gamma $ such that
\begin{equation*}
(g^{\prime },d^{\prime })=(g_{1}gg_{i}^{\prime },g_{1}(gb_{i}^{\prime
}+n^{\prime }d)+b_{1}).
\end{equation*}%
Therefore, we have
\begin{equation}
(g^{\prime },d^{\prime }(\text{mod}(n\Gamma ))\in \mathbb{G}(D_{1}),
\label{mod}
\end{equation}%
since $g_{1}n^{\prime }d\in g_{1}(n^{k^{\ast }}(\Gamma /n^{k^{\ast
}-1}))=n\Gamma $ and $b_{1}\in n\Gamma .$

We shall show the following lemma.

\begin{lemma}
\label{l:finite}There are only finitely many equivalent class generated by
\begin{equation*}
D_{1}=\{(g_{1}^{\prime },b_{1}^{\prime }),\cdots ,\{g_{m^{\prime }}^{\prime
},b_{m^{\prime }}^{\prime }\}\}.
\end{equation*}
\end{lemma}

\begin{proof}
Here $n^{\prime }=n^{k^{\ast }}\ $and $g(B(0,\delta ))=B(0,\delta )$ for any
$g\in \mathbb{G}.$ Then (\ref{mod}) implies that for any $(g^{\prime
},d^{\prime })\in D^{\prime }$ generated by $D_{1},$ there are $g\in \mathbb{%
G}$, $b\in \Gamma $ and index $i\leq m^{\prime }=n^{k^{\ast }}$ such that%
\begin{equation}
g^{\prime }=gg_{i}^{\prime }\text{ and }d^{\prime }=gb_{i}^{\prime }+nb.
\label{auto}
\end{equation}%
As a result, we have
\begin{eqnarray*}
\bigcup\nolimits_{(g^{\prime },d^{\prime })\in D^{\prime }}\left( \frac{%
B(0,\delta )}{n^{\prime }}+d^{\prime }\right)
&=&\bigcup\nolimits_{(g^{\prime },d^{\prime })\in D^{\prime }}\left( \frac{%
B(0,\delta )}{n^{\prime }}+gb_{i}^{\prime }+nb\right) \\
&=&\bigcup\nolimits_{(g^{\prime },d^{\prime })\in D^{\prime }}\left( g(\frac{%
B(0,\delta )}{n^{k^{\ast }}}+b_{i}^{\prime })+nb\right) ,
\end{eqnarray*}%
which is a connected subset of
\begin{eqnarray*}
\mathbb{H}_{k^{\ast }} &=&\bigcup\nolimits_{g\in \mathbb{G}%
}\bigcup\nolimits_{b\in \Gamma }\left( gE_{k^{\ast }}+nb\right) \\
&=&\bigcup\nolimits_{g\in \mathbb{G}}\bigcup\nolimits_{b\in \Gamma
}\bigcup\nolimits_{i=1}^{m^{\prime }}\left( g(\frac{B(0,\delta )}{n^{k^{\ast
}}}+b_{i}^{\prime })+nb\right) .
\end{eqnarray*}

Take $(g^{\prime },d^{\prime })\in D^{\prime }$ satisfying (\ref{auto}), let
\begin{equation*}
D=D^{\prime }-nb\text{ where }b\in \Gamma .
\end{equation*}%
Then $D^{\prime }\sim D$. Notice that
\begin{equation*}
\bigcup\nolimits_{(g,d)\in D}\left( \frac{B(0,\delta )}{n^{\prime }}%
+d\right) =\left( \bigcup\nolimits_{(g^{\prime },d^{\prime })\in D^{\prime
}}\left( \frac{B(0,\delta )}{n^{\prime }}+d^{\prime }\right) \right) -nb
\end{equation*}
is also a connected subset of $\mathbb{H}_{k^{\ast }}$ since $\mathbb{H}%
_{k^{\ast }}=\mathbb{H}_{k^{\ast }}-nb.$

In particular,
\begin{equation*}
g(\frac{B(0,\delta )}{n^{k^{\ast }}}+b_{i})\subset \frac{B(0,\delta )}{%
n^{\prime }}+d^{\prime }-nb\subset \bigcup\nolimits_{(g,d)\in D}\left( \frac{%
B(0,\delta )}{n^{\prime }}+d\right) ,
\end{equation*}%
where
\begin{equation}
g(\frac{B(0,\delta )}{n^{k^{\ast }}}+b_{i}^{\prime })\subset B(0,\frac{%
\delta }{n^{k^{\ast }}}+\delta )\subset B(0,2\delta )  \label{auto1}
\end{equation}%
due to $|b_{i}^{\prime }|\leq \delta .$

By Lemma \ref{l:there} and (\ref{auto1}),
\begin{equation*}
\bigcup\nolimits_{(g,d)\in D}\left( \frac{B(0,\delta )}{n^{\prime }}+d\right)
\end{equation*}
is a connected subset of $B(0,2\delta +1).$ Since $\Gamma $ is discrete and $%
\mathbb{G}$ is finite, we get the finiteness of types generated by $D_{1}.$
\end{proof}

\subsection{Proof of Theorem 3}

$\ $

Lemma \ref{l:finite} insures the \emph{dust-like} graph-directed structure
of fractals related to $E.$ By (\ref{e:decomp}) we have
\begin{equation*}
K_{D}=\frac{1}{n^{\prime }}\bigcup_{D^{\prime }}K_{D^{\prime }}
\end{equation*}%
is a \emph{disjoint} union by the definition of \emph{type}. Here $n^{\prime
}=n^{k^{\ast }}.$

It follows from Lemmas \ref{iso} and \ref{l:finite} that there are dust-like
graph-directed sets $\{K_{D_{1}},K_{D_{2}},\cdots ,K_{D_{p}}\}$ with
\begin{equation*}
D_{1}=\{(g_{1}^{\prime },b_{1}^{\prime }),\cdots ,\{g_{m^{\prime }}^{\prime
},b_{m^{\prime }}^{\prime }\}\}.
\end{equation*}%
For $i\geq 2$, type $D_{i}$ is generated by $D_{1}.$

Suppose the corresponding adjacent matrix is $A=(a_{i,j})_{1\leq i,j\leq p}.$

We shall show
\begin{equation}
m^{\prime }\cdot \#D_{i}=\sum\nolimits_{j=1}^{p}a_{i,j}\#D_{j}.  \label{cha}
\end{equation}

\begin{lemma}
\label{l:subset}For any $D\in \{D_{1},\cdots ,D_{p}\},$ there exist $k$ and
a subset $D^{\prime }$ of
\begin{equation*}
\{(g_{i_{1}}^{\prime }\cdots g_{i_{k}}^{\prime },(n^{\prime
})^{k-1}S_{i_{1}}^{\prime }\cdots S_{i_{k}}^{\prime }(0)\}_{i_{1}\cdots
i_{k}\in \{1,\cdots ,m^{\prime }\}^{k}}
\end{equation*}%
such that $D\sim D^{\prime }$.
\end{lemma}

\begin{proof}
If $D$ is generated by $D_{1}$ directly, then
\begin{equation*}
E=K_{D_{1}}=\frac{1}{n^{\prime }}\bigcup_{i_{1}i_{2}}\left( \frac{1}{%
n^{\prime }}(g_{i_{1}}^{\prime }g_{i_{2}}^{\prime }E)+n^{\prime
}S_{i_{1}}^{\prime }S_{i_{2}}^{\prime }(0)\right) .
\end{equation*}%
Then $D$ is equivalent to a subset of $\{(g_{i_{1}}^{\prime
}g_{i_{2}}^{\prime },(n^{\prime })S_{i_{1}}^{\prime }S_{i_{2}}^{\prime
}(0)\}_{i_{1}i_{2}}.$

Inductively, assume that $D$ is a subset of
\begin{equation*}
\{(g_{i_{1}}^{\prime }\cdots g_{i_{k}}^{\prime },(n^{\prime
})^{k-1}S_{i_{1}}^{\prime }\cdots S_{i_{k}}^{\prime }(0)\}_{i_{1}\cdots
i_{k}\in \{1,\cdots ,m^{\prime }\}^{k}}
\end{equation*}
and $D^{\prime }$ is generated by $D$ directly, we will show that $D^{\prime
}$ is equivalent to a subset of
\begin{equation*}
\{(g_{i_{1}}^{\prime }\cdots g_{i_{k+1}}^{\prime },(n^{\prime
})^{k}S_{i_{1}}^{\prime }\cdots S_{i_{k+1}}^{\prime }(0)\}_{i_{1}\cdots
i_{k+1}\in \{1,\cdots ,m^{\prime }\}^{k+1}}.
\end{equation*}%
For this, we have
\begin{equation*}
K_{D}=\frac{1}{n^{\prime }}\left( \bigcup_{(g,d)\in
D}\bigcup_{i_{k+1}=1}^{m^{\prime }}(\frac{1}{n^{\prime }}(gg_{i_{k+1}}^{%
\prime }E)+gb_{i_{k+1}}^{\prime }+n^{\prime }d)\right) ,
\end{equation*}%
where for $(g,d)=(g_{i_{1}}^{\prime }\cdots g_{i_{k}}^{\prime },(n^{\prime
})^{k-1}S_{i_{1}}^{\prime }\cdots S_{i_{k}}^{\prime }(0))\in D,$%
\begin{equation*}
gg_{i}^{\prime }=g_{i_{1}}^{\prime }\cdots g_{i_{k}}^{\prime
}g_{i_{k+1}}^{\prime },
\end{equation*}%
and
\begin{eqnarray*}
&&gb_{i_{k+1}}^{\prime }+n^{\prime }d \\
&=&g_{i_{1}}^{\prime }\cdots g_{i_{k}}^{\prime }b_{i_{k+1}}+(n^{\prime
})^{k}S_{i_{1}}^{\prime }\cdots S_{i_{k}}^{\prime }(0) \\
&=&(n^{\prime })^{k}\left( \frac{1}{(n^{\prime })^{k}}(g_{i_{1}}^{\prime
}\cdots g_{i_{k}}^{\prime })b_{i_{k+1}}^{\prime }+S_{i_{1}}^{\prime }\cdots
S_{i_{k}}^{\prime }(0)\right) \\
&=&(n^{\prime })^{k}(S_{i_{1}}^{\prime }\cdots S_{i_{k}}^{\prime
})b_{i_{k+1}}^{\prime } \\
&=&(n^{\prime })^{k}(S_{i_{1}}^{\prime }\cdots S_{i_{k+1}}^{\prime })(0).
\end{eqnarray*}%
The inductive process is finished.
\end{proof}

Because $E$ has no complete overlaps, then%
\begin{equation*}
E=\bigcup_{i_{1}i_{2}\cdots i_{k}}\left( \frac{1}{(n^{\prime })^{k}}%
(g_{i_{1}}^{\prime }\cdots g_{i_{k}}^{\prime }E)+S_{i_{1}}^{\prime }\circ
\cdots \circ S_{i_{k}}^{\prime }(0)\right) ,
\end{equation*}%
where
\begin{eqnarray}
&&g_{i_{1}}^{\prime }\cdots g_{i_{k}}^{\prime }=g_{j_{1}}^{\prime }\cdots
g_{j_{k}}^{\prime }\text{ and }S_{i_{1}}^{\prime }\circ \cdots \circ
S_{i_{k}}^{\prime }(0)=S_{j_{1}}^{\prime }\circ \cdots \circ
S_{j_{k}}^{\prime }(0)\text{ }  \label{e:overlap} \\
&&\text{if and only if }i_{1}\cdots i_{k}=j_{1}\cdots j_{k}\in \{1,\cdots
,m^{\prime }\}^{k}.  \notag
\end{eqnarray}
If $D\in \{D_{1},\cdots ,D_{p}\}$ satisfies $K_{D}=\frac{1}{n^{\prime }}%
\bigcup_{D^{\prime }}K_{D^{\prime }},$ by Lemma \ref{l:subset} and (\ref%
{e:overlap}), then
\begin{equation*}
m^{\prime }\cdot \#D=\sum\nolimits_{D^{\prime }}\#D^{\prime }.
\end{equation*}%
Then (\ref{cha}) follows.

Then $\{K_{D_{1}},K_{D_{2}},\cdots ,K_{D_{p}}\}$ are dust-like
graph-directed set with ratio $(n^{\prime })^{-1}(=n^{-k^{\ast }})$ and
integer characteristic $m^{\prime }=m^{k^{\ast }}.$

Therefore, it follows from Theorem 1 that $E=K_{D_{1}}$ is bilipschitz
equivalent to $\Sigma _{m^{\prime }}^{(n^{\prime })^{-1}},$ by Lemma \ref%
{sym}, $E$ is bilipschitz equivalent to $\Sigma _{m}^{n^{-1}}.$

\subsection{Proof of Proposition 1}

$\ $

Let
\begin{equation*}
S_{i}(x)=g_{i}(x)/n+b_{i}.
\end{equation*}

Assume that $E$ has complete overlaps, i.e., there are two sequences $%
i_{1}\cdots i_{t}$ and $j_{1}\cdots j_{t}$ with $i_{1}\neq j_{1}$ such that%
\begin{equation*}
S_{i_{1}}\circ \cdots \circ S_{i_{t}}\equiv S_{j_{1}}\circ \cdots \circ
S_{j_{t}},
\end{equation*}%
that is,
\begin{equation}
g_{i_{1}}\cdots g_{i_{t}}=g_{j_{1}}\cdots g_{j_{t}}  \label{alg1}
\end{equation}%
and $S_{i_{1}}\circ \cdots \circ S_{i_{t}}(0)=S_{j_{1}}\circ \cdots \circ
S_{j_{t}}(0)$ which means%
\begin{equation}
b_{i_{1}}+\frac{g_{i_{1}}b_{i_{2}}}{n}+\cdots +\frac{(g_{i_{1}}\cdots
g_{i_{t-1}})b_{i_{t}}}{n^{t-1}}=b_{j_{1}}+\frac{g_{j_{1}}b_{j_{2}}}{n}%
+\cdots +\frac{(g_{j_{1}}\cdots g_{j_{t-1}})b_{j_{t}}}{n^{t-1}}.
\label{alg2}
\end{equation}%
Notice that (\ref{alg1}) and (\ref{alg2}) imply an edge chain starting at an
original vertex $(b_{i_{1}}-b_{j_{1}},g_{i_{1}},g_{j_{1}})$ and ending at a
boundary vertex $(0,g_{i_{1}}\cdots g_{i_{t}},g_{j_{1}}\cdots g_{j_{t}})$ as
follows
\begin{eqnarray*}
&&(b_{i_{1}}-b_{j_{1}},g_{i_{1}},g_{j_{1}}) \\
&=&(x_{1},g_{1},g_{1}^{\prime })\overset{(i_{2},j_{2})}{\rightarrow }%
(x_{2},g_{2},g_{2}^{\prime })\overset{(i_{3},j_{3})}{\rightarrow }\cdots
\overset{(i_{k},j_{k})}{\rightarrow }(x_{k},g_{k},g_{k}^{\prime }) \\
&&\text{ }\overset{(i_{k+1},j_{k+1})}{\rightarrow }\cdots \overset{%
(i_{t},j_{t})}{\rightarrow }(x_{t},g_{t},g_{t}^{\prime })=(0,g_{i_{1}}\cdots
g_{i_{t}},g_{j_{1}}\cdots g_{j_{t}}),
\end{eqnarray*}%
where $g_{k}=g_{i_{1}}\cdots g_{i_{k}}$ and $g_{k}^{\prime }=g_{j_{1}}\cdots
g_{j_{k}}$ with $g_{t}=g_{t}^{\prime }$ and
\begin{equation}
x_{k+1}=nx_{k}+g_{k}b_{i_{k+1}}-g_{k}^{\prime }b_{j_{k+1}}\in \Gamma .
\label{indu}
\end{equation}%
In fact, we can prove the following equation by induction:
\begin{eqnarray}
&&x_{k}+\frac{1}{n}g_{k}(b_{i_{k+1}}+\frac{g_{i_{k+1}}b_{i_{k+2}}}{n}+\cdots
+\frac{g_{i_{k+1}}\cdots g_{i_{t-1}}b_{i_{t}}}{n^{t-1-k}})  \label{bound} \\
&=&\frac{1}{n}g_{k}^{\prime }(b_{j_{k+1}}+\frac{g_{j_{k+1}}b_{j_{k+2}}}{n}%
+\cdots +\frac{g_{j_{k+1}}\cdots g_{j_{t-1}}b_{j_{t}}}{n^{t-1-k}}).  \notag
\end{eqnarray}%
When $k=1,$ it is true due to (\ref{alg2}). Assume that it is true for $k,$
then (\ref{indu}) implies
\begin{eqnarray*}
&&x_{k+1}+\frac{1}{n}g_{k+1}(b_{i_{k+2}}+\cdots +\frac{g_{i_{k+2}}\cdots
g_{i_{t-1}}b_{i_{t}}}{n^{t-2-k}}) \\
&=&\frac{1}{n}g_{k+1}^{\prime }(b_{j_{k+2}}+\cdots +\frac{g_{j_{k+2}}\cdots
g_{j_{t-1}}b_{j_{t}}}{n^{t-2-k}}).
\end{eqnarray*}%
Furthermore, $x_{t}=0$ since $x_{t}=nx_{t-1}+g_{t-1}b_{i_{t}}-g_{t-1}^{%
\prime }b_{j_{t}}$ and
\begin{equation*}
x_{t-1}+\frac{1}{n}g_{t-1}b_{i_{t}}=\frac{1}{n}g_{t-1}^{\prime }b_{j_{t}}
\end{equation*}%
due to (\ref{bound}). It also follows from (\ref{bound}) that
\begin{equation*}
|x_{k}|\leq 2\max_{i}|b_{i}|(\frac{1}{n}+\frac{1}{n^{2}}+\cdots )\leq
2\max_{i}|b_{i}|/(n-1)=M.
\end{equation*}%
That means $x_{k}\in \{x\in \Gamma :|x|\leq M\}.$

Assume that there exists such path from original vertex to boundary vertex,
in the same way we can prove that $E$ has complete overlaps.

\medskip

The proof of Proposition 1 is completed.

\bigskip

\section{Open questions}

\textbf{Question 1}: How about the rigidity when their characteristics are
\emph{irrational}?

An interesting class of self-similar sets with overlaps is $\{E_{\lambda
}\}_{\lambda },$ where $E_{\lambda}$ is generated by
\begin{equation*}
S_{1}(x)=x/3,S_{\lambda }=x/3+\lambda /3,S_{3}(x)=x/3+2/3,
\end{equation*}%
This class have been studied by Kenyon \cite{K2}, Rao and Wen \cite{RW} and
\'{S}wi\c{a}tek and Veerman \cite{SV}.

Then it is proved in \cite{GLX} that $E_{2/3^{n}}$ and $E_{2/3^{n^{\prime
}}} $ are bilipschitz equivalent for any $n,n^{\prime }\geq 1,$ where $\dim
_{H}E_{2/3^{n}}=\dim _{H}E_{2/3^{n^{\prime }}}=\frac{\log \frac{3+\sqrt{5}}{2%
}}{\log 3}.$

In fact, $E_{2/3^{n}}$ can generate graph-directed sets with adjacent matrix
\begin{equation*}
M_{n}=\left(
\begin{array}{ccccccccc}
1 & 1 & 0 & 0 & \cdots & 0 & 0 & 0 & 0 \\
0 & 0 & 1 & 1 & \cdots & 0 & 0 & 0 & 0 \\
\vdots & \vdots & \vdots & \vdots & \  & \vdots & \vdots & \vdots & \vdots
\\
0 & 0 & 0 & 0 & \cdots & 1 & 1 & 0 & 0 \\
0 & 0 & 0 & 0 & \cdots & 0 & 0 & 1 & 1 \\
1 & 2 & 0 & 0 & \cdots & 0 & 0 & 0 & 0 \\
0 & 0 & 1 & 2 & \cdots & 0 & 0 & 0 & 0 \\
\vdots & \vdots & \vdots & \vdots & \  & \vdots & \vdots & \vdots & \vdots
\\
0 & 0 & 0 & 0 & \cdots & 1 & 2 & 0 & 0 \\
0 & 0 & 0 & 0 & \cdots & 0 & 0 & 1 & 2%
\end{array}%
\right) _{2^{n}\times 2^{n}}.
\end{equation*}%
Here $M_{n}$ has Perron-Frobenius eigenvalue $(3+\sqrt{5})/2$ and
corresponding positive eigenvector.

Our equation is that for $\lambda _{1}=$6/7 and $\lambda _{2}=$8/9, then
their adjacent matrices of corresponding graph-directed sets are
\begin{equation*}
B_{1}=\left(
\begin{array}{cccccc}
1 & 1 & 0 & 0 & 0 & 0 \\
1 & 1 & 1 & 0 & 0 & 0 \\
1 & 0 & 1 & 1 & 0 & 0 \\
1 & 1 & 1 & 0 & 1 & 0 \\
1 & 0 & 1 & 1 & 0 & 1 \\
1 & 0 & 1 & 1 & 0 & 2%
\end{array}%
\right) \text{ and }B_{2}=\left(
\begin{array}{ccc}
1 & 1 & 1 \\
1 & 1 & 2 \\
0 & 1 & 1%
\end{array}%
\right) .\text{ }
\end{equation*}%
Using Mathematica 7.0, we find these two matrices have the same
Perron-Frobenius eigenvalue 2.879$\cdots ,$ we ask
\begin{equation*}
\text{whether }E_{\lambda _{1}}\text{ and }E_{\lambda _{2}}\text{ are
bilipschitz equivalent or not?}
\end{equation*}%
\begin{figure}[tbph]
\centering\includegraphics[width=0.6\textwidth]{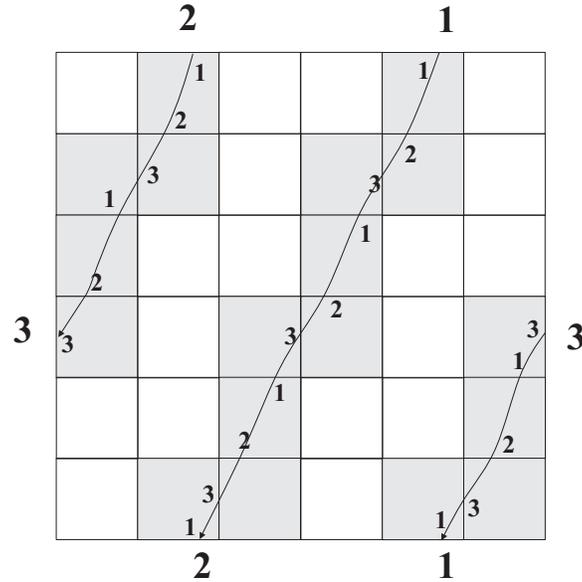} \vspace{-0.5cm}
\caption{The self-similar set which is not totally disconnected}
\end{figure}

\bigskip

Looking at Theorem 3, we find there is an algorithm to test the existence of
complete overlaps, then the following question arise naturally.

\textbf{Question 2}: How to test in polynomial time the total
disconnectedness of self-similar set in $\Lambda ,$ specially for
self-similar set in the form (\ref{total})?

The total disconnectedness seems to be difficult to test. Please see the
following interesting example.

\begin{example}
The initial self-similar pattern in Figure $7$ insures this self-similar set
includes infinitely many curves.

In the unit cube, we have $\gamma _{1}$ from placement $2$ to placement $3,$
$\gamma _{2}$ from placement $1$ to placement $2$ and $\gamma _{3}$ from
placement $3$ to placement $1$. In the small squares with sidelength $1/6$,
we also have small curves which are similar to $\gamma _{1},\gamma
_{2},\gamma _{3}$ respectively. Therefore, this self-similar set includes $%
\gamma _{1},\gamma _{2},\gamma _{3}$.

In fact, $\{\gamma _{1},\gamma _{2},\gamma _{3}\}$ are graph-directed sets $%
( $satisfying the open set condition$)$ with ratio $1/6$ and adjacent matrix
\begin{equation*}
\left(
\begin{array}{ccc}
2 & 2 & 1 \\
3 & 3 & 3 \\
1 & 1 & 2%
\end{array}%
\right) .
\end{equation*}
\end{example}

\bigskip
  \textbf{Additional note.} Later on, the authors~\cite{XiXi13} further studied the
  problem on the Lipschitz equivalence of self-similar sets satisfying the OSC. By
        some different ideas, we can prove a more general form of Theorem 3 but
  at the cost of tedious complications.

\end{document}